\documentclass{birkjour}
\usepackage{amssymb}
\usepackage{amsmath}
\usepackage{amsthm} 
\usepackage{mathtools}
\usepackage{color}
\usepackage[svgnames]{xcolor}
\usepackage[utf8]{inputenc}
\usepackage[T1]{fontenc}
\usepackage{graphicx}
\usepackage{placeins}
\usepackage{vmargin}
\usepackage{setspace}
\usepackage{stmaryrd}

\numberwithin{equation}{section}
\numberwithin{figure}{section}

\newtheorem{theorem}{Theorem}[section]
\newtheorem{proposition}[theorem]{Proposition}
\newtheorem{lemma}[theorem]{Lemma}

\newtheorem{corollary}[theorem]{Corollary}

\theoremstyle{definition}
\newtheorem{definition}[theorem]{Definition}
\newtheorem{property}[theorem]{Property}
\newtheorem{assumptions}{Assumptions}
\newtheorem{assumption}[assumptions]{Assumption}
\theoremstyle{remark}
\newtheorem{remark}[theorem]{Remark}
\newtheorem{claim}[theorem]{Claim}

\renewcommand{\theenumi}{\roman{enumi}}
\newcommand{\norm}[1]{\left\lVert #1\right\rVert}

\newcommand{\eps}{\varepsilon}

\newcommand{\R}{\mathbb{R}}

\newcommand{\FFF}{\mathcal{F}}
 
\newcommand{\JJJ}{\mathcal{J}}

\newcommand{\PPP}{\mathcal{P}}
\newcommand{\RRR}{\mathcal{R}}

\newcommand{\tvarphi}{\tilde{\varphi}}
\newcommand{\tw}{\tilde{w}}

\newcommand{\supent}[1]{\left\lceil #1 \right\rceil}
\renewcommand{\leq}{\leqslant}
\renewcommand{\geq}{\geqslant}

\DeclareMathAlphabet{\mathpzc}{OT1}{pzc}{m}{it}

\renewcommand{\Im}{\mathcal I\!\mathpzc{m}}

\DeclareMathOperator{\loc}{loc}

\begin{document}

\title[Double power NLS]{Mass-energy scattering criterion for double power Schr\"odinger equations}

\author[Thomas Duyckaerts]{Thomas Duyckaerts}
\address[Thomas Duyckaerts]{LAGA (UMR 7539),
\newline\indent
Institut Galil\'ee, Universit\'e Sorbonne Paris Nord,
  \newline\indent
  99 avenue Jean-Baptiste Cl\'ement,
  \newline\indent
  93430 Villetaneuse,
  France
  \newline\indent
  and
  Département Math\'ematiques et Application,
  \newline\indent
  \'Ecole Normale Sup\'erieure
  \newline\indent
  45 rue d'Ulm
  \newline\indent
  75005 Paris,
  France}
\email[Thomas Duyckaerts]{duyckaer@math.univ-paris13.fr}

\author[Phan Van Tin]{Phan Van Tin}
\address[Phan Van Tin]{LAGA (UMR 7539),
\newline\indent
Institut Galil\'ee, Universit\'e Sorbonne Paris Nord,
\newline\indent
  99 avenue Jean-Baptiste Cl\'ement,
  \newline\indent
  93430 Villetaneuse,
  France}
\email[Phan Van Tin]{vantin.phan@math.univ-paris13.fr}


\date{\today}

\begin{abstract} 
We consider the nonlinear Schr\"odinger equation with double power nonlinearity. We extend the scattering result in \cite{KiOhPoVi17} for all $L^2$-supercritical powers, specially, our results adapt to the cases of energy-supercritical nonlinearity. 
\end{abstract}

\maketitle

\tableofcontents
\section{Introduction and main resuls}

This article concerns the nonlinear Schr\"odinger (NLS) equation
\begin{equation}
 \label{NLS}
 i\partial_tu+\Delta u=|u|^{p_0}u-|u|^{p_1}u,
\end{equation} 
in space dimension $d\geq 1$, where $\frac{4}{d}<p_1< p_0$. 

The equation \eqref{NLS} has $3$ conserved quantities: the mass,
\begin{equation*}
M(u(t))=\int_{\R^d} |u(t,x)|^2dx,
\end{equation*}
the energy
\begin{equation*}
E(u(t))=\int_{\R^d}|\nabla u(t,x)|^2dx+\frac{2}{p_0+2}\int_{\R^d}|u(t,x)|^{p_0+2}dx-\frac{2}{p_1+2}\int_{\R^d}|u(t,x)|^{p_1+2}dx.
\end{equation*}
and the momentum
\begin{equation*}
P(u(t))=\Im \int_{\R^d}\nabla u(t,x)\,\overline{u}(t,x)\,dx.
\end{equation*}
We recall that a solution of \eqref{NLS} with initial data in a Sobolev space $H^{s}$ where the equation \eqref{NLS} is well-posed is said to scatter in $H^{s}$ forward in time when it is defined on $[0,\infty)$, and there exists $v_0\in H^{s}$ such that
$$\lim_{t\to\infty} \left\|e^{it\Delta}v_0-u(t)\right\|_{H^{s}}=0.$$
Our goal is to give sufficient conditions of scattering for solutions of \eqref{NLS}.



The equation \eqref{NLS} has stationary wave solutions, of the form $e^{i\omega t}Q_{\omega}$, where $Q_{\omega}$ is a solution of the equation
\begin{equation}
 \label{Ell}
 -\Delta \varphi=-|\varphi|^{p_0}\varphi+|\varphi|^{p_1}\varphi-\omega \varphi,
\end{equation}
see \cite[Theorem 2]{LeNo20} and references therein. It
has positive solutions for some values of $\omega$. More precisely, there exists an explicit $\omega_*$ (depending only on $p_0$ and $p_1$) such that for $0<\omega<\omega_*$, the equation \eqref{Ell} has exactly one positive solution (up to translation) $Q_{\omega}\in (H^1\cap L^{\infty})(\R^d)$, and no such solution for $\omega\geq \omega_*$.

Some of these solutions are related to the minimization problem:
\begin{equation}
 \label{minim}
I(m)=\inf_{\begin{matrix} \varphi \in H^1(\R^d) \cap L^{p_0+2}(\R^d)\\ \int_{\R^d} |\varphi|^2\,dx =m \end{matrix}} E(u).
 \end{equation}
 (see \cite[Theorem 6]{LeNo20}).
There is a critical mass $m_c$ such that for $m\leq m_c$, $I(m)=0$ and for $m>m_c$, $I(m)<0$. The minimum is not attained for $m<m_c$. In the case $p_1>\frac{4}{d}$ that we consider here, the minimum is attained for $m\geq m_c$ by a positive, radial solution of \eqref{Ell}, that we will denote by $Q_{\omega}$.  The solution $Q_{\omega_c}$ corresponding to the mass $m_c$ has minimal mass among all solutions of the family of equations \eqref{Ell}.

Obviously, the stationary wave solutions do not scatter. Other examples of nonscattering solutions are travelling waves that are exactly the Galilean transforms of stationary waves. According to the soliton resolution conjecture, these stationary and travelling waves should be the only obstruction, leading to the conjecture that solutions with mass below $m_c$ scatter in both time directions. A weaker result was proved in \cite{KiOhPoVi17} in the case $p_0=4$, $p_1=2$, $d=3$, namely that scattering holds for solutions with mass $M(u)\leq \frac{4}{3\sqrt{3}}\tilde{m}_c$, and with an additional bound on the energy in the range of mass $\frac{4}{3\sqrt{3}}\tilde{m}_c<M(u)<m_c$. The goal of this article is to show that this picture is not specific to the case studied in \cite{KiOhPoVi17}, but holds for all exponents $\frac{4}{d}<p_1<p_0$. We will rely on the work \cite{LeNo20} on the stationary problem \eqref{Ell}, and on our previous article \cite{DuyckaertsPhan24P}, where preliminary results were obtained for a family of nonlinear Schr\"odinger equations that includes \eqref{NLS}.

As in \cite{KiOhPoVi17}, we will use the standard compactness/rigidity scheme (as initiated in \cite{KeMe06} for the energy-critical problem). One of the key ingredient of this scheme is the positivity of the functional:
$$\Phi(u)=\int |\nabla u|^2-\frac{dp_1}{2(p_1+2)}|u|^{p_1+2}+\frac{dp_0}{2(p_0+2)}|u|^{p_0+2},$$
that appears in the virial identity
\begin{equation}
\label{virial_ID}
\frac{d}{dt} \Im \int x\cdot\nabla u\,\overline{u}=2\Phi(u).
\end{equation}
In Section \ref{S:var}, we prove that there is a critical value $\tilde{m}_c\in(0,m_c)$ such that $\Phi(u)>0$ for every nonzero $u$ such that $M(u)<\tilde{m}_c$, and that for $\tilde{m}_c\leq M(u)<m_c$, we have $\Phi(u)>0$ with an additional condition on the energy. Namely, let:
\begin{equation}
\label{def_tilde_m}
\tilde{m}_c=\left( \frac{p_0}{p_1} \right)^{\frac{p_1d-4}{2(p_0-p_1)}}\left( \frac{4}{dp_1} \right)^{\frac d2}m_c.
\end{equation}
Then:
\begin{theorem}
\label{T:positivity}
Assume $4/d<p_1<p_0$.
There exists a nonincreasing function $\mathsf{e}(m)$: $[0,m_c)\to (0,\infty]$, with $\mathsf{e}(m)=\infty$ for $0<m<\tilde{m}_c$, and $\mathsf{e}(m)\in (0,\infty)$ for $\tilde{m}_c\leq m<m_c$, such that, letting
\begin{equation}
\label{defRRR}
 \RRR:= \left\{\varphi\in H^{s_0}\cap H^1,\; 0< \int|\varphi|^2<m_c,\; E(\varphi)<\mathsf{e}\Big(\int |\varphi|^2\Big)\right\},
\end{equation}
where $s_0=\frac{d}{2}-\frac{2}{p_0}$, one has
$$\forall u\in \RRR,\quad \Phi(u)>0.$$
The function $\mathsf{e}$ is strictly decreasing on $[\tilde{m}_c,m_c]$ if $d\in \{1,2,3,4\}$ or $p_0\leq \frac{4}{d-2}$.
\end{theorem}


We next state our scattering results. Throughout of this paper, we will make the following additional assumption on the exponent:
\begin{assumption}
\label{Assum:NL}
$\frac{4}{d}<p_1< p_0$ and for all $j\in\{1,2\}$, $\supent{s_0}< p_j$ or $p_j$ is an even integer.\footnote{$\supent{s}$ denotes the smallest integer number larger or equal $s$.}
\end{assumption}
This type of assumption is classical to ensure a minimal regularity of the nonlinearity, however we have not tried to obtain the optimal conditions. It is certainly possible to improve this condition using techniques as in \cite{Visan07}.

We first consider the case where the higher order nonlinearity is energy-subcritical or energy-critical, that is $d=1,2$ or $d\geq 3$ and $p_0\leq \frac{4}{d-2}$. Equivalently, the critical Sobolev exponent $s_0=\frac{d}{2}-\frac{2}{p_0}$ for the usual (NLS) equation with a single power nonlinearity
\begin{equation}
 \label{NLSh}
 i\partial_tu+\Delta u=|u|^{p_0}u
\end{equation}
satisfies $s_0\leq 1$.  We recall that this equation \eqref{NLSh} is well-posed in $H^1$ (see \cite{Kato87}, \cite{CaWe90}),

One can prove in this case that the equation is well-posed in $H^1$. It follows from conservation laws and the standard Gagliardo-Nirenberg inequality that all solutions of \eqref{NLS} are bounded in $H^1$. Using the well-posedness theory when $s_0<1$, and a more refined argument ultimately relying on
the fact that all solutions of \eqref{NLSh} scatter (see \cite{CoKeStTaTa08}, \cite{RyVi07} and \cite{Visan07}) for $s_0=1$, one can deduce from this bound that all solutions of \eqref{NLS} are global (see \cite{Zhang06}, \cite{TaViZh07}).
\begin{theorem}
\label{T:scatt_intro1}
Assume $d=1,2$, or $d\geq 3$ and $p_0\leq \frac{4}{d-2}$. Suppose that Assumption \ref{Assum:NL} holds. Let $u_0\in \RRR$. Then $u$ is global and scatters in both time directions.
\end{theorem}
Theorem \ref{T:scatt_intro1} was proved in the case $d=3$, $p_0=4$ (thus $s_0=1$), $p_1=2$ in \cite{KiOhPoVi17}. See also \cite[Theorem 1.3]{TaViZh07} where a scattering result for small mass was obtained for general double power-nonlinearities with $s_0\leq 1$.

We note that the ground state solution $e^{it\omega_c}Q_{\omega_c}$ provides a non-scattering solution of \eqref{NLS} which is bounded in the space $H^{s_0}\cap H^1$, with mass $m_c$. However, as in \cite{KiOhPoVi17}, the first critical mass $\tilde{m}_c$ is not given by this minimal ground state mass, but by a multiple of it (corresponding to the mass of a rescaled solution of \eqref{Ell}). It is not clear whether this restriction is due the method, based on the positivity of $\Phi(u)$, or if there exists nonscattering, bounded solutions of \eqref{NLS} with $\tilde{m}_c<\int |u_0|^2<m_c$.

In the case $s_0>1$, the equations \eqref{NLS} and \eqref{NLSh} are no longer well-posed in $H^1$, but rather in the Sobolev space $H^{s_0}$ (see \cite[Section 2]{DuyckaertsPhan24P}), at least under Assumption \ref{Assum:NL}.  The homogeneous equation \eqref{NLSh} is also well-posed in the homogeneous Sobolev space $\dot{H}^{s_0}$ (see \cite{CaWe90} and \cite[Remark 2.20]{DuyckaertsPhan24P}).


Furthermore the global well-posedness is not guaranteed. Indeed Merle, Rapha\"el, Rodnianski and Szeftel \cite{MeRaRoSz22a} have constructed solutions of the homogeneous equation \eqref{NLSh} for some values of $p_0$, $d$ (with $s_0>1$), that blow up in finite time. It is very likely that such a phenomenon persists in the case of a double power nonlinearity.

It is known however that for many values of $p_0$ a solution of \eqref{NLSh} that remains bounded in the critical Sobolev space are global and scatter.
We thus consider the following property:
\begin{property}
\label{Proper:bnd}
Let $A_0\in (0,\infty]$.
For any solution $u$ of \eqref{NLSh} with initial data in $\dot{H}^{s_0}$, if
\begin{equation}
  \label{bound_Hs0}
  \limsup_{t\to T_{+}(u)} \|u(t)\|_{\dot{H}^{s_0}}<A_0.
\end{equation}
Then $T_{+}(u)=+\infty$ and $u$ scatters for positive times in $\dot{H}^{s_0}$, i.e. there exists $v_0\in \dot{H}^{s_0}$ such that
$$\lim_{t\to\infty} \left\|u(t)-e^{it\Delta}v_0\right\|_{\dot{H}^{s_0}}=0.$$
\end{property}
Property \ref{Proper:bnd} holds for small $A_0$ by standard small data theory. It is conjectured that it holds for $A_0=\infty$. It was proved
\cite{CoKeStTaTa08}, \cite{RyVi07} and \cite{Visan07} in the case $s_0=1$ (where the bound \eqref{bound_Hs0} is automatically given by conservation of the energy). The study of Property \ref{Proper:bnd} for other critical exponents was initiated in \cite{KeMe10} where it was proved when $d=3$, $p_0=2$ (thus $s_0=1/2$), for radial solutions. It was later proved in many other cases:  see e.g. \cite{KillipVisan10}, \cite{MiXuZh14}, \cite{LuZheng17}, \cite{DoMiMuZh17} for $s_0>1$ in dimension $d\geq 4$ (with technical restriction if $d\geq 7$).
\begin{theorem}
\label{T:scatt_intro2}
Assume $d\geq 3$, $p_0>\frac{4}{d-2}$ and Assumption \ref{Assum:NL} holds.
Let $A_0>0$ such that Property \ref{Proper:bnd} holds. Let $u$ be a solution of \eqref{NLS} such that $u_0\in \RRR$ and $u$ satisfies \eqref{bound_Hs0}. Then $u$ is global and scatters forward in time.
\end{theorem}


The proofs of Theorem \ref{T:scatt_intro1} and \ref{T:scatt_intro2} is by contradiction. One of the tool needed for this method is a profile decomposition adapted to equation \eqref{NLS}, which was constructed in \cite{DuyckaertsPhan24P} in a more general context. Using this profile decomposition, one can show that if the scattering does not hold, there exists a nonzero global solution $u_c$, of \eqref{NLS} such that there exists $x(t)$ such that
$$K=\left\{u_c(t,\cdot+x(t)),\; t\in \R\right\}$$
has compact closure in $H^{s_0}\cap H^1$, and
$\forall t,\quad u_c(t)\in \RRR$.
The second step of the argument is to prove that $0$ is the only solution $u_c$ (with mass and energy in the considered domain) such that $K$ defined above has compact closure in $H^1\cap H^{s_0}$. For this we use a general rigidity results, proved in \cite{DuyckaertsPhan24P}, and that extends the standard rigidity result (see \cite{KeMe06} for radial (NLS) equation, \cite{DuHoRo08} for extension to nonradial cases, and \cite{KiOhPoVi17} for the cubic/quintic double power nonlinearity):
\begin{proposition}\label{pro1}
With the assumptions above, let $u$ be a solution of \eqref{NLS} defined on $[0,\infty)$ such that there exists $x(t)$ defined for $t\in [0,\infty)$ with
\begin{equation}
\label{defK}
K=\{u(t,x+x(t)); t\geq 0\}
\end{equation}
has compact closure in $H^{s_0}\cap H^1$. Then
\begin{equation}
\label{eqB12}
\min_{t\geq 0}\left|\Phi(u(t)) - \frac{|P(u)|^2}{M(u)}\right|=0.
\end{equation}
\end{proposition}


To conclude the proof of Theorems \ref{T:scatt_intro1} and \ref{T:scatt_intro2}, it remains to prove that
$$\Phi(u) - \frac{|P(u)|^2}{M(u)}>0,$$
for $u\in \RRR$. This follows easily from Theorem \ref{T:positivity} considering all the $e^{ix\cdot\xi_0}u$ and optimizing in $\xi_0$.

 The mathematical study of (NLS) equation with a double-power nonlinearity was initiated in \cite{Zhang06}, in dimension $3$, where the author investigated the global well-posedness, scattering and blow-up phenomena in the case $p_0=4$. This includes in particular a scattering result for small mass, in the spirit of Theorem \ref{T:scatt_intro2}, but without explicit mass. Similar results were obtained in \cite{TaViZh07}, in general dimension $d$ in the energy-critical and supercritical setting $s_0\leq 1$. As already mentioned, Theorem \ref{T:scatt_intro2} was proved in \cite{KiOhPoVi17} in the particular case $d=3$, $p_0=4$, $p_1=2$. The scattering was extended to a slightly larger region in \cite{KillipVisan21}. In \cite{MiXuZh13}, the authors showed that solutions of a specific (NLS) scatter in both time directions by giving a new radial profile decomposition. 

The problem with $p_1=\frac{4}{d}$, $p_0<\frac{4}{d-2}$ was considered in \cite{Cheng20}, \cite{Murphy21}, where the author investigated scattering below or at the mass of the ground-state for the mass-critical homogeneous equation. See also \cite{CaCh21} which considers the
case $(p_0,p_1)=(4,2)$ in space dimensions $1$, $2$ and $3$.


We conclude this introduction with an outline of the paper. In Section \ref{S:preliminary}  we recall results from \cite{DuyckaertsPhan24P} in the particular case of equation \eqref{NLS}. This includes a well-posedness theory, as well as a profile decomposition adapted to \eqref{NLS} and a quantitative version of Property \ref{Proper:bnd}. In Section \ref{S:var}, we prove Theorem \ref{T:positivity} on the positivity of the virial functional. In Section \ref{S:scattering} we prove our scattering results Theorems \ref{T:scatt_intro1} and \ref{T:scatt_intro2}.

\section*{Acknowledgment}

The second author is supported by Post-doc fellowship from Labex MME-DII: SAIC/2022 No 10078.

\section{Preliminaries}
\label{S:preliminary}
\subsection{Cauchy theory}
This section is concerned with the Cauchy and stability theory for the equation \eqref{NLS}. We refer to \cite{DuyckaertsPhan24P} for the proofs.
We start with a few notations.

For each $q\geq 1$, we let $q'$ be the conjugate exponent given by
$
\frac{1}{q}+\frac{1}{q'}=1.
$

If $m$ is a complex valued function on $\R^d$, we define by $m(\nabla)$ the Fourier multiplier with symbol $m(\xi)$, i.e. $\widehat{m(\nabla )u}=m(\xi)\hat{u}(\xi)$, where $\hat{u}$ is the Fourier transform of $u$.
For a multi-index $\alpha=(\alpha_1,\alpha_2,...,\alpha_d)$, denote
$
D^{\alpha}=\partial_{x_1}^{\alpha_1}\cdot\cdot\cdot \partial_{x_d}^{\alpha_d}$ and $|\alpha|=\sum_{i=1}^d|\alpha_i|.$

If $X$ is a vector space, $(u,v)\in X^2$, we will make a small abuse of notation, denoting $\|(u,v)\|_{X}=\|u\|_X+\|v\|_X$.

We will consider the following function spaces:
$$S^0(I)=\begin{cases}
L^{\infty}\left(I,L^2(\R^d)\right)\cap L^{2}\left(I,L^{\frac{2d}{d-2}}(\R^d)\right)& \text{ if }d\geq 3\\
L^{\infty}\left(I,L^2(\R^2)\right)\cap L^{q_2}\left(I,L^{r_2}(\R^2)\right)& \text{ if }d=2 \\
L^{\infty}\left(I,L^2(\R)\right)\cap L^{4}\left(I,L^{\infty}(\R)\right)& \text{ if }d=1,
         \end{cases}
$$
where when $d=2$, $(q_2,r_2)$ is an admissible pair with $q_2>2$ close to $2$, and 
\begin{equation*}
N^0(I)= \begin{cases}
L^{1}\left(I,L^2(\R^d)\right)+ L^{2}\left(I,L^{\frac{2d}{d+2}}(\R^d)\right)& \text{ if }d\geq 3\\
L^{1}\left(I,L^2(\R^2)\right) + L^{q_2'}\left(I,L^{r_2'}(\R^2)\right)& \text{ if }d=2 \\
L^{1}\left(I,L^2(\R)\right)+ L^{4/3}\left(I,L^{1}(\R)\right)& \text{ if }d=1,
         \end{cases}
\end{equation*}
By Strichartz estimates (\cite{Strichartz77}, \cite{GiVe85} \cite{KeTa98}), 
$$ \|u\|_{S^0(I)}\lesssim \|u_0\|_{L^2}+\|f\|_{N^0(I)}$$
for any solution $C^0(I,\R)$ of 
$i\partial_tu +\Delta u=f$ on $I\ni 0$ with initial data $u_0$ at $t=0$.

We also define the following Strichartz spaces, and dual Strichartz spaces:
\begin{equation*}
W^0(I)=
L^{\frac{2(2+d)}{d}}\left(I\times\R^d\right)
 \quad Z^0(I)=\left(W^0(I)\right)'=L^{\frac{2(d+2)}{d+4}}(I\times \R^d),
\end{equation*}
so that we have $S^0(I)\subset W^0(I)$ and $Z^0(I)\subset N^0(I)$ (with continous embedding). We denote
\begin{equation}
\label{defX}
X(I)=L^{p_0(d+2)/2}_{t,x}(I\times\R^d)\cap L^{p_1(d+2)/2}_{t,x}(I\times\R^d).
\end{equation}
For $s\geq 0$, we denote
\begin{gather*}
\norm{u}_{S^s(I)}=\norm{\langle\nabla\rangle^{s} u}_{S^0(I)},\quad \norm{u}_{\dot{S}^s(I)}=\norm{|\nabla |^{s} u}_{S^0(I)}
\end{gather*}
and define similarly $W^s(I)$, $\dot{W}^s(I)$,  $N^s(I)$, $\dot{N}^s(I)$, $Z^s(I)$ and $\dot{Z}^s(I)$.

By Sobolev inequalities, $S^{s_0}(I)$ is continuously embedded in $X(I)$.


\begin{definition}
Let $0\in I$ be an interval. By definition, a solution $u$ to \eqref{NLS} on $I$, with initial data in $H^{s_0}$ is a function $u\in C(I,H^{s_0}) $ such that for all $K \subset I$ compact, $u\in S^{s_0}(K)$ and $u$ satisfies the following Duhamel formula
\begin{equation}
 \label{Duhamel}
u(t)=e^{it\Delta}u_0-i\int_0^t e^{i(t-\tau)\Delta}g(u)(\tau)\,d\tau,
 \end{equation} 
for all $t\in I$, where $g(u)=|u|^{p_0}u-|u|^{p_1}u$.
\end{definition}
Noting that the assumption $p_0>\frac{4}{d}$ implies $p_0(d+2)/2>2$ (and $p_0(d+2)/2\geq 6$ if $d=1$), we can choose $q_0$ such that $(p_0(d+2)/2,q_0)$ is an admissible pair. By Sobolev inequality and the definitions of $s_0$, $q_0$, one can check
\begin{equation}
 \label{inclusion}
 \|u\|_{L^{\frac{p_0(d+2)}{2}}_x}\lesssim \left\| |\nabla |^{s_0} u\right\|_{L^{q_0}}
 \end{equation} 
and thus
\begin{equation}
 \label{Ss_X}
 \|u\|_{L^{\frac{p_0(d+2)}{2}}(I\times \R^d)}\lesssim \||\nabla |^{s_0}u\|_{S^{0}(I)},\quad \|u\|_{X(I)}\lesssim \|u\|_{S^{s_0}(I)}.
\end{equation}

The following local well-posedness statement is contained in \cite[Section 2]{DuyckaertsPhan24P}. See in particular Proposition 2.12, Remarks 2.13, 2.14 and 2.15 and Lemma 2.16 there:
\begin{theorem}\label{T:local wellposed}
Let $u_0\in H^{s_0}$. Let $p_0$, $p_1$ such that Assumption \ref{Assum:NL}  holds. Then there exists a unique maximal solution of $u$ to \eqref{Duhamel}, defined on an interval $(T_-,T_+)\ni 0$, such that for any interval $I$ such that $\overline{I}\subset (T_-,T_+)$, one has $u\in S^{s_0}(I)$.
Furthermore we have the following:
\begin{itemize}
 \item $T_+<\infty \Longrightarrow u\notin X([0,T_{+})).$
 \item If $u\in X([0,\infty))$, then $u$ scatters as $t\to\infty$.
\end{itemize}
\end{theorem}

\begin{remark}
Let $I$ be an interval such that $0\in I$. If $\norm{e^{i\cdot\Delta}u_0}_{X(I)}<\delta$ is small enough, then $u$ is defined on $I$. By Strichartz, $\norm{e^{i\cdot\Delta}u_0}_{X(\R)}\lesssim \norm{u_0}_{H^{s_0}}<\infty$ which implies local well posedness of \eqref{Duhamel}.
\end{remark}
We next state the existence of wave operators for equation \eqref{NLS}:
\begin{proposition}\label{P:wave operators}
Let $u_0\in H^{s_0}$. Let $p_0$, $p_1$ such that Assumption \ref{Assum:NL}  holds. Let $v_0\in H^{s_0}$ and $v_L(t)=e^{it\Delta}v_0$. Then there exist a unique solution $u\in C^0((T_{-}(u),\infty),H^{s_0})$ of \eqref{NLS} such that
$$\lim_{t\to \infty}\|u(t)-v_L(t)\|_{H^{s_0}}=0.$$
\end{proposition}


We recall the following long time perturbation theory:
\begin{theorem}(see \cite{DuyckaertsPhan24P}[Theorem 2.19])\label{T:long time perturbation}
Let $A>0$. There exists constants $\eps(A)\in (0,1]$, $C(A)>0$ with the following properties.
Let $0\in I$ be an compact interval of $\R$ and $w$ be a solution of the following equation
\[
Lw=g(w)+e,
\] 
and $u_0\in H^{s_0}$ such that 
\begin{equation*}
\norm{w_0}_{H^{s_0}}+\norm{w}_{X(I)}\leq A,\quad
\norm{e}_{N^{s_0}(I)}+
\norm{u_0-w(0)}_{H^{s_0}}=\varepsilon\leq \eps(A).
\end{equation*}
Then the solution $u$ of \eqref{NLS} with initial data $u_0$ is defined on $I$ and satisfies
\begin{equation}
\label{bounduw}
\norm{u-w}_{S^{s_0}(I)}\leq C(A) \eps,\quad
\norm{u}_{S^{s_0}(I)}\leq C(A).
\end{equation}
\end{theorem}


\begin{remark}
 \label{R:homogeneous}
The analog of Theorem \ref{T:local wellposed} for the equation \eqref{NLSh} where all the spaces are replaced by homogeneous spaces hold. Precisely, in the statement of Theorem \ref{T:local wellposed}, one can replace $X(I)$ by $L^{\frac{p_0(d+2)}{2}}(I\times \R^d)$, $H^{s_0}$ by $\dot{H}^{s_0}$, $S^{s_0}(I)$ by $\dot{S}^{s_0}(I)$.
\end{remark}

\subsection{Profile decomposition}
\label{sub:profile}
In this subsection we recall results on profile decomposition for the linear Schr\"odinger equation
 \begin{equation}
 \label{LS}
 i\partial_tu_L+\Delta u_L=0.
\end{equation}
and the nonlinear Schr\"odinger equation \eqref{NLS}. Profile decompositions for Schr\"odinger equations go back to \cite{MeVe98} (see also, among other works, \cite{Keraani01}, \cite{BeVa07} and \cite{Shao09}). Here, we follow the theory developed in Section 3 of our previous article \cite{DuyckaertsPhan24P} which is adapted to nonlinearities of the form $\pm |u|^{p_0}u+l.o.t.$.

We start by recalling the profile decomposition for sequences of solutions of \eqref{LS} such that the corresponding of  initial data is bounded in $H^{s_0}$.
We denote by $u_L(t)=e^{it\Delta}u_0$ the solution to this equation, with initial data  $u_L(0)=u_0\in H^{s_0}$
\begin{definition}
 \label{D:hom_profile}
 A linear $H^{s_0}$-\emph{profile}, in short \emph{profile}, is a sequence $(\varphi_{Ln})_n$, of solutions of \eqref{LS}, of the form 
 \begin{equation}
  \label{linear_profile}
 \varphi_{Ln}(t,x)=\frac{1}{\lambda_n^{\frac{2}{p_0}}}\varphi_L\left( \frac{t-t_n}{\lambda_n^2},\frac{x-x_n}{\lambda_n} \right), 
\end{equation}
 where $\varphi_L$ is a fixed solution of \eqref{LS} and $\Lambda_n=(\lambda_n,t_n,x_n)_n$ is a sequence in $(0,1]\times \R\times \R^d$ (called sequence of transformations) such that
 \begin{equation}
  \label{lim_time}
  \lim_{n\to\infty}\frac{-t_n}{\lambda_n^2}=\tau\in \R\cup \{\pm\infty\}.
 \end{equation} 
\end{definition}
\begin{definition}
\label{D:ortho_equiv}
We say that two sequence of transformations $\Lambda_n=(\lambda_n,t_n,x_n)$ and $M_n=(\mu_n,s_n,y_n)$ are \emph{orthogonal} when they satisfy
$$\lim_{n\to\infty} \frac{|t_n-s_n|}{\lambda_n^2}+\frac{|x_n-y_n|}{\lambda_n}+\left|\log\left( \frac{\lambda_n}{\mu_n} \right)\right|=\infty.$$ 
 We say that two $H^{s_0}$-profiles are \emph{orthogonal} when one of the two profiles is identically $0$ or when the corresponding sequences of transformations are orthogonal.
\end{definition}

Let $q_0=\frac{2d}{d-2s_0}=\frac{d}{2}p_0$ be the Lebesgue exponent such that the Sobolev embedding $\dot{H}^{s_0}\subset L^{q_0}$ holds. Then we have:
\begin{proposition}
\label{P:decomposition1}
For any bounded sequence $(u_{0,n})_n$ in $H^{s_0}$, there exists a subsequence (that we still denote by $(u_{0,n})_n$), 
a sequence $\left(\Big(\varphi_{Ln}^j\Big)_{n}\right)_{j\geq 1}$ of $H^{s_0}$ profiles such that 
\begin{equation}
 \label{limitwnJ}
\lim_{J\to\infty}\limsup_{n\to\infty}\left\|w_{Ln}^J\right\|_{X(\R)}+\left\| |\nabla|^{s_0}w_{Ln}^J\right\|_{W^0(\R)}+\|w_{Ln}^J\|_{L^{\infty}L^{q_0}}=0,
\end{equation} 
where 
\begin{equation}
\label{def_wLnJ} 
w_{Ln}^J=u_{Ln}-\sum_{j=1}^J \varphi_{Ln}^j,
\end{equation} 
$u_{L,n}=e^{it\Delta}u_{0,n}$.
 Furthermore, we have the Pythagorean expansions:
\begin{gather}
 \label{Pythagorean1}
 \forall J\geq 1,\quad \|u_{0,n}\|^2_{\dot{H}^{s_0}}=\sum_{j=1}^J \left\|\varphi_{L}^j(0)\right\|^2_{\dot{H}^{s_0}}+\left\|w_{Ln}^J(0)\right\|^2_{\dot{H}^{s_0}}+o(1),\quad n\to\infty\\
 \label{PythagoreanHs}
 \forall J\geq 1,\quad \|u_{0,n}\|^2_{\dot{H}^{s}}=\sum_{\substack{j\in \JJJ_{NC}\\1\leq j\leq J}} \left\|\varphi_{L}^j(0)\right\|^2_{\dot{H}^{s}}+\left\|w_{Ln}^J(0)\right\|^2_{\dot{H}^{s}}+o(1),\quad n\to\infty;\; 0\leq s<s_0.
\end{gather} 
We say that the sequence $\left((\varphi^j_{Ln})_n\right)_{j\geq 1}$ is a profile decomposition of $(u_{Ln})_n$
\end{proposition}
See Definition 3.4, Proposition 3.5, Lemma 3.7 in \cite{DuyckaertsPhan24P}, and also Remark 3.10 there for the Pythagorean expansion \eqref{PythagoreanHs}. The norm $X(\R)$ does not appear in the analog of \eqref{limitwnJ} in \cite{DuyckaertsPhan24P} (see (3.5) there). The fact that this norm also goes to $0$ follows from the boundedness of the sequence $(u_{0,n})_n$ is bounded in $H^{s_0}$ and Strichartz estimates.

We next construct a nonlinear analog, adapted to the equation \eqref{NLS}, of the preceding profile decomposition.
With the notations of Proposition \ref{P:decomposition1}, extracting subsequences if necessary, we see that for all $j$, one of the following holds:
\begin{itemize}
\item $\varphi^j_L\equiv 0$. In this case we say that $(\varphi_{Ln}^j)_n$ is a \emph{null profile}, and denote $j\in \JJJ_0$. 
\item $\varphi^j_L\not \equiv 0$ and $\lim_{n\to\infty}\lambda_{n}^j=\lambda^j\in (0,1]$. In this case we say that $(\varphi_{Ln}^j)_n$ is a \emph{non-concentrating profile}, and denote $j\in \JJJ_{NC}$. We have $\varphi_0^j=\varphi^j_L(0)\in H^{s_0}$.
Rescaling, we see that we can assume $\lambda^j=1$. As a consequence, we can assume $\lambda_n^j=1$ for all $n$.
\item $\varphi^j_L \not\equiv 0$ and $\lim_{n\to\infty}\lambda_n^j=0$. In this case we say that $(\varphi_{Ln}^j)_n$ is a \emph{concentrating profile}, and denote $j\in\JJJ_C$. 
\end{itemize}
\begin{remark}
\label{R:noconcentration}
 If the sequence $(u_{0,n})_n$ is bounded in $H^{s}$ for some $s>s_0$, then $\JJJ_C=\emptyset$.
\end{remark}
\begin{remark}
\label{R:Lq}
Let $2<q<q_0=\frac{2d}{d-2s_0}$, and 
 \begin{equation}
 \label{deftauj}
 \tau^j=\lim_{n\to\infty} \frac{-t_n^j}{(\lambda_n^j)^2}.
\end{equation}
Extracting subsequences and translating the profiles in time, we can assume $\tau^j\in\{0,\pm\infty\}$. Then if $\tau^j\in \{\pm\infty\}$ or $j\in \JJJ_C$
\begin{equation}
\label{0Lq}
\lim_{n\to\infty}\|\varphi_{Ln}^j(0)\|_{L^q}=0.
\end{equation} 
Furthermore, we have the following expansion:
$$\lim_{n\to\infty}\|u_{0,n}\|_{L^q}^{q}=\sum_{\substack{j\in \JJJ_{NC}\\ \tau^j=0}}\|\varphi_{L}^j(0)\|^{q}_{L^q}.$$ 
\end{remark}
To each (linear) profile $\varphi_{Ln}^j$, we associate a \emph{nonlinear profile} $\varphi_{n}^j$ and a \emph{modified nonlinear profile} $\tilde{\varphi}_{n}^j$  in the following way:
\begin{itemize}
\item If $j\in \JJJ_0$, $\tilde{\varphi}^j_n$ and $\varphi_n^j$ are both equal to the constant null function.
 \item If $j\in \JJJ_{NC}$, the 
 modified  nonlinear profile and the nonlinear profile are equal, and defined by 
 $$\tvarphi_n^j(t,x)=\varphi_n^j(t,x)=\varphi^j\left(t-t^j_n,x-x_n^j\right),$$
 where  $\varphi^j$ is the unique solution of \eqref{NLS} such that 
 \begin{equation}
 \lim_{t\to\tau^j}\left\|\varphi^j(t)-\varphi^j_L(t)\right\|_{H^{s_0}}=0,\quad \tau^j=\lim_{n\to\infty}-t^j_n.
 \end{equation}
 \item If $j\in \JJJ_{C}$, the nonlinear profile $\varphi_n^j$ is defined by 
 \begin{equation}
 \label{defphinj}
 \varphi_{n}^j(t,x)=\frac{1}{(\lambda_n^j)^{\frac{2}{p_0}}}\varphi^j\left( \frac{t-t_n^j}{(\lambda_n^j)^2},\frac{x-x_n^j}{\lambda_n^j} \right)
 \end{equation} 
 where $\varphi^j$ is the unique solution of the homogeneous equation \eqref{NLSh} such that
 $$\lim_{t\to\tau^j} \left\|\varphi^j(t)-\varphi_L^j(t)\right\|_{\dot{H}^{s_0}}=0,\quad \tau^j=\lim_{n\to\infty}\frac{-t_{j,n}}{(\lambda_n^j)^2}.$$
  To define the modified profile, we fix $\sigma^j$ in the maximal interval of existence of $\varphi^j$ (if $\tau^j$ is finite, we can take $\sigma^j=\tau^j$, if $\tau^j=\pm\infty$, $|\sigma^j|$ large and with the same sign than $\tau^j$). We let 
 $$ s^j_n=(\lambda_n^j)^2\sigma^j+t_n^j$$
 and denote by $\tilde{\varphi}_n^j$ the solution of \eqref{NLSh} such that 
 \begin{equation}
 \label{def_tildephij}
 \tilde{\varphi}_n^j(s^j_n)=\chi\left( x-x_n^j \right)\varphi_n^j(s_n^j)= \chi\left(x-x_n^j\right)\frac{1}{(\lambda_n^j)^{\frac{2}{p_0}}}\varphi^j\left( \sigma^j,\frac{x-x_{n}^j}{\lambda_{n}^j} \right),  
 \end{equation} 
where $\chi\in C_0^{\infty}(\R^N)$ is radially symmetric, $\chi(x)=1$ for $|x|<1$, $\chi(x)=0$ for $|x|>2$.
 \end{itemize}
When $j\in \JJJ_C$, the modified profiles $\tilde{\varphi}_n^j$ are close to the profiles $\varphi_n^j$ in $\dot{H}^{s_0}$, \begin{lemma}
\label{L:modifprofile1}
 Let $j\in \JJJ_C$. Then
 \begin{equation}
  \label{equiv_prof1}
  \lim_{n\to\infty}\left\|\tilde{\varphi}^j_n(0)-\varphi_{n}^j(0)\right\|_{\dot{H}^{s_0}}=0.
 \end{equation}
 \end{lemma}
These modified profiles are also approximate solutions of \eqref{NLS} (see Lemma 3.13 in \cite{DuyckaertsPhan24P}).

We next give an approximation result. We let $\tw_{L,n}^J$ be the solution of the linear wave equation with initial data
\begin{equation}
\label{def_tilde_w}
\tilde{w}_{L,n}^J(0)=u_{0,n}-\sum_{j=1}^J\tilde{\varphi}_n^j(0).
\end{equation} 
One can prove that for all $s$ with $0<s\leq s_0$,
 $$\lim_{J\to\infty}\limsup_{n\to\infty}\left\|\tilde{w}_{L,n}^J\right\|_{X(\R)}+\left\|\tilde{w}_{L,n}^J\right\|_{\dot{W}^s(\R)}=0.$$
Then we have:
 \begin{theorem}[Approximation by profiles]
\label{T:NLapprox}
Let $(u_{0,n})_n$ be a sequence bounded in $H^{s_0}$ that admits a profile decomposition $\left((\varphi^j_{Ln})_n\right)_{j}$. Let $\varphi_n^j$,  $\tvarphi_n^j$ and $\tilde{w}_{L,n}^J$ be as above. Let $I_n$ be a sequence of intervals such that $0\in I_n$, and assume that for each $j\geq 1$, for large $n$,
$0\in I_n \subset I_{\max}(\varphi_n^j),$
\begin{gather}
\label{scatt_profile1}
j\in \JJJ_C\Longrightarrow \limsup_{n\to\infty} \left\||\nabla|^{s_0}\varphi_n^j\right\|_{S^{0}(I_n)}<\infty,\\
\label{scatt_profile2}
j\in \JJJ_{NC}\Longrightarrow \limsup_{n\to\infty} \left\|\varphi_n^j\right\|_{S^{s_0}(I_n)}<\infty.
\end{gather}
Let $u_n$ be the solution of \eqref{NLS} with initial data $u_{0,n}$. Then for large $n$, $I_n\subset I_{\max}(u_n)$, 
$$ u_n(t)=\sum_{1\leq j\leq J} \tvarphi_n^j(t,x)+\tilde{w}_{L,n}^J(t)+r_{n}^J(t),\quad t\in I_n,$$
with 
$$\limsup_{n\to \infty}\|u_n\|_{S^{s_0}(I_n)}<\infty$$
and
\begin{equation}
 \label{lim_rnJ}
\lim_{J\to\infty}\limsup_{n\to\infty}\left\|r_n^J\right\|_{S^{s_0}(I_n)}=0.
 \end{equation} 
\end{theorem}

We will also need the fact that Pythagorean expansions of the Sobolev norms hold in the setting of the preceding Theorem:
\begin{lemma}
\label{L:Pythagorean}
With the same assumptions and notations as in Theorem \ref{T:NLapprox}, if $(t_n)_n$ is a sequence of time with $t_n\in I_n$ for all time, then for all $J\geq 1$,
$$ \|u_n(t_n)\|^2_{\dot{H}^{s_0}}=\sum_{j=1}^J\|\tilde{\varphi}^j_n(t_n)\|_{\dot{H}^{s_0}}^2+\|\tilde{w}_n^J(t_n)\|^2_{\dot{H}^{s_0}}+o_{J,n}(1),$$
where $\lim_{J\to\infty}\limsup_{n\to\infty} o_{J,n}(1)=0$. Furthermore, for all $s$ with $0\leq s<s_0$, 
$$ \|u_n(t_n)\|^2_{\dot{H}^{s}}= \sum_{\substack{1\leq j\leq J\\ j\in \JJJ_{NC}}}\|\tilde{\varphi}^j_n(t_n)\|_{\dot{H}^{s}}^2+\|\tilde{w}_n^J(t_n)\|^2_{\dot{H}^{s}}+o_{J,n}(1),$$
\end{lemma}

\subsection{A uniform bound on the scattering norm}
We recall here \cite[Proposition 4.1]{DuyckaertsPhan24P}, which gives a quantitative version of Property \ref{Proper:bnd}. The proof uses profile decomposition:
\begin{proposition}
\label{P:boundh}
Let $p_0>\frac{4}{d}$, $s_0=\frac{d}{2}-\frac{2}{p_0}$. Assume Property \ref{Proper:bnd}. Then for all $A\in (0,A_0)$, there exists $\FFF(A)>0$ such that for all interval $0\in I$, for all solution $u\in C^0(I,\dot{H}^{s_0})$ of \eqref{NLSh} such that
\begin{equation}
\label{Hs0_bound}
\sup_{t\in I}\|u(t)\|_{\dot{H}^{s_0}}\leq A
\end{equation} 
we have 
\begin{equation}
\label{s0_bound}
\|u\|_{\dot{S}^{s_0}(I)}\leq \FFF(A).
\end{equation} 
\end{proposition}

\subsection{Global well posedness} 
Using Proposition \ref{P:boundh}, we obtain that if Property \ref{Proper:bnd} holds then solutions satisfying \eqref{bound_Hs0} is global forward in time. More precisely, we have the following result:
\begin{theorem}(see \cite[Theorem 1.2]{DuyckaertsPhan24P})
\label{T:GWP}
Let $p_0>\frac{4}{d}$, $s_0=\frac{d}{2}-\frac{2}{p_0}$. Suppose that Assumption \ref{Assum:NL} holds. Assume Property \ref{Proper:bnd}. Let $u$ be a solution of \eqref{NLS}, with initial data in $H^{s_0}$, such that \eqref{bound_Hs0} holds. Then $T_+(u)=\infty$.
\end{theorem}

\section{Variational properties}
\label{S:var}
In this section, we fix $d\geq 1$ and consider a nonlinearity $g$ of the form
$$ g(u)=|u|^{p_0}u-|u|^{p_1}u,$$
with $\frac{4}{d}<p_1<p_0,$ such that $\frac{d}{2}-\frac{2}{p_0}=s_0\geq 1$. We also assume that $p_0$ and $p_1$ are even integers, or that $\supent{s_0}\leq p_1$. 

We recall the definition of the energy and the virial functional:
\begin{align*}
E(u)&=\int_{\R^d} |\nabla u|^2-\frac{2}{p_1+2}|u|^{p_1+2}+\frac{2}{p_0+2}|u|^{p_0+2}\,dx,\\
\Phi(u)&=\int_{\R^d} |\nabla u|^2-\frac{dp_1}{2(p_1+2)}|u|^{p_1+2}+\frac{dp_0}{2(p_0+2)}|u|^{p_0+2}\,dx,
\end{align*}
In this subsection, we prove positivity properties of $E$ and $\Phi$ in some subsets of $H^{s_0}$, see Theorems \ref{thm1}, Corollary \ref{cor2} and Theorem \ref{thm3} below. These properties are obtained essentially as consequences of \cite{LeNo20}.


\subsection{Preliminaries}

As in \cite{LeNo20}, we have the following Gagliardo-Nirenberg inequality
\begin{equation}
\label{eq4}
\norm{u}^{p_1+2}_{p_1+2}\leq_{p_1,p_0,d} \norm{u}_{L^2}^{p_1-\theta p_0}\norm{\nabla u}_{L^2}^{2(1-\theta)}\norm{u}^{\theta(p_0+2)}_{p_0+2},
\end{equation}
where 
\[
\theta=\frac{p_1-\frac{4}{d}}{p_0-\frac{4}{d}} \in [0,1).
\] 
We note that this implies the following:
\begin{claim}
\label{Cl:nabla}
Let $u\in H^1(\R^d)\cap L^{p_0+2}(\R^d)$ be such that $M(u)=m>0$ and $\Phi(u)=0$ Then,
\[
\norm{\nabla u}^2_2 \gtrsim_{d,p_1,p_0,m} 1. 
\]
\end{claim}
\begin{proof}
Using $\Phi(u)=0$ and Gagliardo-Nirenberg inequality \eqref{eq4}, we have
\[
\frac{dp_0}{2(p_0+2)}\norm{u}^{p_0+2}_{p_0+2}+\norm{\nabla u}^2_2 =\frac{dp_1}{2(p_1+2)}\norm{u}^{p_1+2}_{p_1+2}\lesssim_{d,p_1,p_0,m}\norm{\nabla u}^{2(1-\theta)}_2\norm{u}^{\theta(p_0+2)}_{p_0+2}. 
\]
Thus, $\norm{u}^{p_0+2}_{p_0+2}\approx \norm{u}^{p_1+2}_{p_1+2}\approx \norm{\nabla u}^2_2$. If $p_1<\frac{4}{d-2}$ then 
\[
\norm{u}_{p_1+2}\lesssim \norm{u}_2^{1-\frac{dp_1}{2(p_1+2)}}\norm{\nabla u}_2^{\frac{dp_1}{2(p_1+2)}}.
\]
Then, 
\[
\norm{\nabla u}_2^2 \approx\norm{u}_{p_1+2}^{p_1+2}\lesssim \norm{\nabla u}_2^{\frac{dp_1}{2}}.
\]
It implies that $\norm{\nabla u}_2 \gtrsim_{d,p_1,p_0,m} 1$ since $p_1> 4/d$.\\
If $p_1=\frac{4}{d-2}$ then 
\[
\norm{\nabla u}_2^{\frac{2}{p_1+2}}\lesssim \norm{u}_{p_1+2}\lesssim \norm{\nabla u}_2.
\]
It implies that $\norm{\nabla u}_2 \gtrsim_{d,p_1,p_0,m} 1$. \\
If $p_1>\frac{4}{d-2}$ then 
\[
\norm{u}_{p_1+2}\lesssim \norm{u}_{\frac{2d}{d-2}}^{\theta}\norm{u}_{p_0+2}^{1-\theta}\lesssim \norm{\nabla u}_2^{\theta}\norm{u}_{p_0+2}^{1-\theta},
\]
where $\theta\in (0,1)$ satisfies 
\[
\frac{1}{p_1+2}=\frac{\theta}{2d/(d-2)}+\frac{1-\theta}{p_0+2}.
\] 
Thus,
\[
\norm{\nabla u}_2^{\frac{2}{p_1+2}}\lesssim \norm{\nabla u}_2^{\theta}\norm{\nabla u}_2^{\frac{2(1-\theta)}{p_0+2}}.
\]
We see that $\frac{2}{p_1+2}<\theta+\frac{2(1-\theta)}{p_0+2}$. Thus, $\norm{\nabla u}_2 \gtrsim_{d,p_1,p_0,m} 1$. This completes the proof.
\end{proof}

Define
\[
I(m)=\inf_{\begin{matrix} u \in H^1(\R^d) \cap L^{p_0+2}(\R^d)\\ \int_{\R^d} |u|^2\,dx =m \end{matrix}} E(u).
\]
The following Theorem is essentially contained in \cite{LeNo20}.
\begin{theorem}\label{thm1}
Let $d \geq 1$ and $p_0>p_1>\frac{4}{d}$. The function $m \mapsto I(m)$ is concave non-increasing over $[0,\infty)$. There exists $m_c>0$ such that it satisfies
\begin{itemize}
\item[•] $I(m)=0$ for all $0\leq m\leq m_c$,
\item[•] $m\mapsto I(m)$ is negative and strictly decreasing on $(m_c,\infty)$. 
\end{itemize}
The problem $I(m)$ admits at least one positive, radial, decreasing minimizer $u$ for every $m\geq m_c$.
Any minimizer $u$ solves the Euler-Lagrange equation 
\begin{equation}\label{eq3.1}
-\Delta u=-u^{p_0+1}+u^{p_1+1}-\omega u,
\end{equation}
for some $\omega \in (0,\omega_*)$, hence must be equal to $Q_{\omega}$, where $\omega_*,Q_{\omega}$ are defined in the introduction. The minimizer in the case $m=m_c$ is an optimizer for the Gagliardo-Nirenberg inequality \eqref{eq4}.

The infimum is not attained for $m<m_c$.  

For $0<m<m_c$, there exists $\delta(m)>0$ such that for any $u \in H^1(\R^d)\cap L^{p_0+2}(\R^d)$ such that $\int_{\R^d} |u|^2\,dx=m < m_c$, one has $E(u)\geq \delta(m)\|\nabla u\|^2_2$. 

For each minimizer $u$ of the problem $I(m)$ satisfies the Pohozaev identity: $\Phi(u)=0$.
\end{theorem}
\begin{proof}
 In the case $d\geq 2$, this is \cite[Theorem 6]{LeNo20}, except for the claim that there exists $\delta(m)$ such that $E(u)\geq \delta(m)\|\nabla u\|^2_{2}$ if $\int |u|^2dx=m<m_c$. To prove this point, we let $v(x)=u(\beta x)$, where $\beta^d=\frac{m}{m_c}<1$.
Then $\int |v|^2=m_c$ and $E(v)=\frac{m_c}{m}((\beta^2-1) \norm{|\nabla u|}^2_2+E(u))$.
Since $E(v)\geq 0$, we obtain $E(u)\geq (1-\beta^2)\norm{\nabla u}^2_2$. 

Let $u$ be a minimizer of $I(m)$. Then $u$ solves the following Euler-Lagrange equation:
$$-\Delta u= -|u|^{p_0}u+|u|^{p_1}u-\lambda u,$$
for some $\lambda\in\R$. By Claim \ref{Cl:Pohozaev}, we have $\Phi(u)=0$.

It remains to consider the case $d=1$. The existence and uniqueness of solution for elliptic equation \eqref{eq3.1} are proved in \cite{BeLi83a}, \cite[Proposition 2.1]{LiTsZw21}, \cite[Theorem 8.1.4]{Ca03}. Thus, we only need to show the existence of such $m_c$ in the case $d=1$. It is classical that $I$ is non-positive and non-increasing. We show that $I$ is continuous on $(0,\infty)$. Indeed, as in \cite[Proof of Theorem 6]{LeNo20}, we can rewrite 
    \[
    I(m)=m J(m^{-2}),
    \]
    where 
    \[
    J(\lambda):=\inf_{\begin{matrix} u \in H^1(\R) \cap L^{p_0+2}(\R)\\ \int_{\R} |u|^2\,dx =1 \end{matrix}} H_{\lambda}(u),
    \]
    where 
    \[
    H_{\lambda}(u):=\frac{\lambda}{2}\int |\partial_x u|^2 +\frac{1}{p_0+2}\int |u|^{p_0+2}-\frac{1}{p_1+2}\int |u|^{p_1+2}.
    \]
    It suffices to show that $J$ is continuous on $(0,\infty)$. Let $\lambda \in (0,\infty)$ and $\lambda_n\rightarrow \lambda$ as $n\rightarrow\infty$. Let $u_n$ be such that $M(u_n)=1$, $J(\lambda_n) +2^{-n}\geq H_{\lambda_n}(u_n)$. It is clear that $u_n$ is uniformly bounded in $H^1(\R)\cap L^{p_0+2}(\R)$. Thus, for $n$ large,
    \[
    J(\lambda_n)+2^{-n}\geq H_{\lambda_n}(u_n)\geq H_{\lambda}(u_n)+o(1)\geq J(\lambda)+o(1)
    \]
    which implies that
    \[
    J(\lambda)\leq \liminf_{n\rightarrow\infty}J(\lambda_n).
    \]
    Thus, it suffices to show that
    \[
    J(\lambda) \geq \limsup_{n\rightarrow\infty}J(\lambda_n).
    \]
    Let $\eps>0$, $u\in H^1(\R)\cap L^{p_0+2}(\R)$ be such that $M(u)=1$, $J(\lambda)\geq H_{\lambda}(u)-\eps$. We have, for $n$ large enough,
    \[
    J(\lambda)\geq H_{\lambda}(u)-\eps=H_{\lambda_n}(u)-\eps+o(1) \geq J(\lambda_n)-\eps+o(1),
    \]
    which implies the desired result, letting $n\to\infty$ then $\eps\to 0$.\\
    
    Next, we show that there exists $m_c>0$ such that $I(m_c)=0$. We argue by contradiction. Assume that $I(m)<0$ for all $m>0$. We fix $m>0$ close to $0$ and consider a minimizing sequence $(u_n)_n$, with $u_n\in H^1(\R)$, $\int |u_n|^2=m$ and $\lim_n E(u_n)=I(m)$. Using Schwarz rearrangement we can assume $u_n$ nonnegative, even, and decreasing for $x>0$.  By the inequality $u^{p_0+2}+Cu^2\geq u^{p_1+2}$, we see that $(u_n)_n$ is bounded in $H^1$. Extracting subsequences, we can assume that $(u_n)_n$ converges weakly to some $u\in H^1$. The convergence is strong in $L^{p_0+2}\cap L^{p_1+2}$. Thus 
    \[ E(u)\leq I(m),\quad m'=\int |u|^2\leq m\]
    Since $I$ is nonincreasing, we deduce that $I(m')\leq E(u)\leq I(m)\leq I(m')$. Thus $u$ is a minimizer for $I(m')$. If $m'=0$, then the equality $I(m)=I(m')=0$ shows that $I=0$ close to $0$ and we are done. If $m'>0$,
$u$ is a positive minimizer of $I(m')$. Let $\omega$ be the associated Lagrange multiplier. Then $\omega \in (0,\omega_*)$ and
    \[
    \partial_{xx} u - u^{p_0+1}+u^{p_1+1}-\omega u=0.
    \]
    This implies that
    \begin{equation}
    \label{eq_ENS}
        \int |\partial_x u|^2 + \omega \int |u|^2=\int |u|^{p_1+2}-\int |u|^{p_0+2}.
    \end{equation}
    where all integrals take over the real line. 
    By the Gagliardo-Nirenberg type inequality \eqref{eq4}  and since $m'\leq m$, we have
    \[ \int |u|^{p_1+2}\leq Cm^{\frac{p_1-\theta p_0}{2}} \|\partial_x u\|_{L^2}^{2(1-\theta)} \|u\|_{L^{p_0+2}}^{\theta(p_0+2)}.\]
    Combining with Young's inequality $a^{\theta}b^{1-\theta}\leq \theta a+(1-\theta)b$, we obtain
    \[ \int |u|^{p_1+2}\leq \frac{1}{2}\int |u|^{p_0+2}+Cm^{\frac{p_1-\theta p_0}{2(1-\theta)}}\int (\partial_xu)^2.\]
    With \eqref{eq_ENS}, we deduce
    \[ \int |\partial_xu|^2 \leq Cm^{\frac{p_1-\theta p_0}{2(1-\theta)}}\int (\partial_xu)^2 -\frac{1}{2}\int |u|^{p_0+2}.\] 
    This gives a contradiction if $m$ is small enough. Hence, there exists $m_c>0$ such that $I=0$ on $(0,m_c)$ and $I(m)<0$ if $m>m_c$. It remains to prove that $I$ is strictly decreasing on $(m_c,\infty)$. Assume by contradiction that there exist $m_c<m_1<m_2$ such that $I(m_1)=I(m_2)$. 
    
    Let $m\in (m_1,m_2)$ and $u_n$ be a minimizer sequence of $I(m)$. We can assume that $u_n$ is positive, radial, decreasing. It is classical that $u_n$ is uniformly bounded in $H^1(\R) \cap L^{p_0+2}(\R)$ and hence $u_n$ converges weakly to $u$ in $H^1(\R)$ and strongly in $L^q$ for each $q\in (2,\infty)$. Thus, 
    \begin{align*}
        E(u) &\leq \liminf_{n\rightarrow \infty} E(u_n) = I(m),\\
        M(u) &\leq  \liminf_{n\rightarrow \infty} M(u_n)=m,
    \end{align*}
    and hence, $u$ is a minimizer of the following problem
    \[
    \inf \{ E(v): v\in H^1(\R)\cap L^{p_0+2}(\R), M(v)\in (0,m_2)\}.
    \]
    Thus, $u$ is a critical point of $E$:
    \[
    \partial_{xx} u - u^{p_0+1}+ u^{p_1+1}=0.
    \]
    This shows that
    \begin{align*}
        \int |\partial_x u|^2&=\int |u|^{p_1+2}-\int |u|^{p_0+2},\\
        -\int |\partial_x u|^2 &=\frac{2}{p_1+2}\int |u|^{p_1+2}-\frac{2}{p_0+2}\int |u|^{p_0+2}.
    \end{align*}
    Hence,
    \begin{align*}
        \frac{p_1+4}{p_1+2}\int |u|^{p_1+2}&=\frac{p_0+4}{p_0+2}\int |u|^{p_0+2},\\
        2\int |\partial_x u|^2 &= \frac{p_1}{p_1+2}\int |u|^{p_1+2}-\frac{p_0}{p_0+2}\int |u|^{p_0+2}\\
        &=\int |u|^{p_0+2} \frac{4p_1-4p_0}{(p_0+2)(p_1+2)}<0,
    \end{align*}
    which gives a contradiction. Hence, $I$ is strictly decreasing on $(m_c,\infty)$ and the proof is completed.
\end{proof}
Recall from \eqref{def_tilde_m} the definition of $\tilde{m}_c$. In the remainder of this subsection, we will prove:
\begin{theorem}
 \label{thm3}
 There exists a nonincreasing function $\mathsf{e}:(0,m_c)\to [0,\infty]$ such that $\mathsf{e}(m)=\infty$ if $0< m< \tilde{m}_c$, and, for all $u\in H^1(\R^d)\cap L^{p_0+2}(\R^d)$ such that $M(u)=m\in (0,m_c)$, $0<E(u)<\mathsf{e}(m)$, one has
$$\Phi(u)-|P(u)|^2/M(u)>0.$$
\end{theorem}
\begin{remark}
\label{R:decreasing}
If $d\in \{1,2,3,4\}$ or $d\geq 5$ and $p_0\leq \frac{4}{d-2}$, the function $e$ is strictly decreasing on $[\tilde{m}_c,m_c]$. The case $d\geq 5$, $p_0>\frac{4}{d-2}$ is open.
\end{remark}

\subsection{Positivity of virial functional below critical mass}
Define
\[
\tilde{I}(m)=\inf_{\begin{matrix} u \in H^1(\R^d) \cap L^{p_0+2}(\R^d)\\ \int_{\R^d} |u|^2\,dx =m \end{matrix}} \Phi(u).
\]
Let $u \in H^1(\R^d) \cap L^{p_0+2}(\R^d)$ and $a,b>0$. Define
\[
u_{a,b}(x)=au(bx).
\]
We have 
\begin{equation*}
\int_{\R^d}|u_{a,b}(x)|^2\,dx=a^2b^{-d}\int_{\R^d}|u|^2\,dx,\quad
E(u_{a,b})
=a^2b^{2-d} E_{a,b}(u),
\end{equation*}
where
\[
E_{a,b}(u):=\int_{\R^d}|\nabla u|^2\,dx-\frac{2}{p_1+2}a^{p_1}b^{-2}\int_{\R^d}|u|^{p_1+2}\,dx+\frac{2}{p_0+2}a^{p_0}b^{-2}\int_{\R^d}|u|^{p_0+2}\,dx.
\]
We have
\begin{align*}
I(m)&=\inf_{\begin{matrix} u \in H^1(\R^d) \cap L^{p_0+2}(\R^d)\\ \int_{\R^d} |u_{a,b}(x)|^2\,dx =m \end{matrix}}E(u_{a,b})=a^2b^{2-d}\inf_{\begin{matrix} u \in H^1(\R^d) \cap L^{p_0+2}(\R^d)\\ \int_{\R^d} |u|^2\,dx =m a^{-2}b^d \end{matrix}} E_{a,b}(u),
\end{align*}
for all $a,b>0$.\\
We chose $a=a_0$, $b=b_0$ such that
\begin{equation*}
\frac{2}{p_1+2}a_0^{p_1}b_0^{-2}=\frac{dp_1}{2(p_1+2)},\quad
\frac{2}{p_0+2}a_0^{p_0}b_0^{-2}=\frac{dp_0}{2(p_0+2)},
\end{equation*}
that is
\begin{equation*}
a_0= \left(\frac{p_0}{p_1}\right)^{\frac{1}{p_0-p_1}},\quad
b_0^2=\frac{4}{dp_1}\left(\frac{p_0}{p_1}\right)^{\frac{p_1}{p_0-p_1}}.
\end{equation*}
By the definition of $a_0,b_0$, $E_{a_0,b_0}\equiv \Phi$. By the above, we have
\[
I\left(a_0^2b_0^{-d}m\right)=a_0^2b_0^{2-d}\inf_{\begin{matrix} u \in H^1(\R^d) \cap L^{p_0+2}(\R^d)\\ \int_{\R^d} |u|^2\,dx =m \end{matrix}} \Phi(u)=a_0^2b_0^{2-d} \tilde{I}(m).
\]
Define $\tilde{m}_c=(a_0^2b_0^{-d})^{-1}m_c$.

We have
\begin{equation*}
\tilde{m}_c<m_c\Leftrightarrow a^2_0>b_0^d\Leftrightarrow  \left(\frac{p_0}{p_1}\right)^{\frac{2}{p_0-p_1}}>\left(\frac{4}{dp_1}\right)^{\frac{d}{2}}\left(\frac{p_0}{p_1}\right)^{\frac{dp_1}{2(p_0-p_1)}}
\Leftrightarrow \left(\frac{p_0}{p_1}\right)^{\frac{p_1}{p_0-p_1}}<\left(\frac{dp_1}{4}\right)^{\frac{dp_1}{dp_1-4}}
\end{equation*}
\begin{equation}
\tilde{m}_c<m_c\iff f_1\left(\frac{p_0-p_1}{p_1}\right)<f_2\left(\frac{dp_1}{4}\right), \label{eq78}
\end{equation}
where $f_1(x)=(1+x)^{\frac{1}{x}}$ and $f_2(x)=x^{\frac{x}{x-1}}$. We see that $f_1$ is strictly decreasing on $(0,\infty)$ and $f_2$ is strictly increasing on $(1,\infty)$. Thus, 
\[
f_1\left(\frac{p_0-p_1}{p_1}\right)<\lim_{x\rightarrow 0}f_1(x)=e=\lim_{x\rightarrow 1}f_2(x)<f_2\left(\frac{dp_1}{4}\right)
\]
Then, \eqref{eq78} holds when $p_0>p_1>4/d$, and $\tilde{m}_c<m_c$.

From Theorem \ref{thm1}, we have the following result.
\begin{corollary}\label{cor2}
Let $d \geq 1$ and $p_0>p_1>\frac{4}{d}$. The function $m \mapsto \tilde{I}(m)$ is concave non-increasing over $[0,\infty)$. It satisfies
\begin{itemize}
\item[•] $\tilde{I}(m)=0$ for all $0\leq m\leq \tilde{m}_c$,
\item[•] $m\mapsto \tilde{I}(m)$ is negative and strictly decreasing on $(\tilde{m_c},\infty)$. The problem $\tilde{I}(m)$ admits at least one positive, radial, decreasing minimizer $\tilde{u}$ for every $m\geq \tilde{m}_c$.
\end{itemize}
For $0<m<\tilde{m}_c$, there exists $\tilde{\delta}(m)$ such that 
for any $u\in L^{p_0+2}\cap H^1$ with $\int |u|^2=m$,
$$\Phi(u)\geq \tilde{\delta}(m)\|\nabla u\|^2_{L^2}.$$
\end{corollary}

\subsection{Positivity of virial functional below a critical energy, at fixed mass}
As before, we assume $\frac{4}{d}<p_1<p_0$ and thus $\tilde{m}_c<m_c$. Define
\[
I^{\Phi}(m)=\inf_{\begin{matrix} u \in H^1(\R^d) \cap L^{p_0+2}(\R^d)\\ M(u)=m,\,\Phi(u)=0 \end{matrix}} E(u).
\]
If there is no function $u \in H^1(\R^d) \cap L^{p_0+2}(\R^d)$ such that $M(u)=m$ and $\Phi(u)=0$ then we let $I^{\Phi}=\infty$. \\
Define
\[
\mathcal{R}=\{(m,e): 0<m<m_c,\, 0<e<I^{\Phi}(m)\}.
\]
As in \cite[Theorem 5.2]{KiOhPoVi17}, we have the following result.
\begin{theorem}
\label{thm90}
If $(M(u),E(u))\in \mathcal{R}$ for some $u\in H^1(\R^d) \cap L^{p_0+2}(\R^d)$ then $\Phi(u)>0$.
\begin{itemize}
\item $I^{\Phi}(m)=\infty$ when $0<m<\tilde{m}_c$.
\item $0<I^{\Phi}(m)<\infty$ when $\tilde{m}_c\leq m<m_c$.
\item $I^{\Phi}(m)=I(m)$ when $m\geq m_c$.
\end{itemize}
Moreover, $I^{\Phi}(m)$ is nonincreasing as a function of $m$ for $m\geq \tilde{m}_c$. 

If we further assume that $d=1,2,3,4$ or $p_0\leq \frac{4}{d-2}$, then the infimum $I^{\Phi}(m)$ achieved. Moreover, in these cases, $I^{\Phi}(m)$ is strictly decreasing and lower semicontinous on $[\tilde{m}_c,\infty)$.
\end{theorem}

Before proving Theorem \ref{thm90}, we prove that it is equivalent to consider the minimum of $E$ on the set $\{\Phi(u)=0\}$ and on the larger set $\{\Phi(u)\leq 0\}$.
\begin{lemma}
\label{lm90}
Suppose $u\in H^1(\R^d) \cap L^{p_0+2}(\R^d)$ is not identically zero and that \label{I:Phi<0} $\Phi(u)<0$
Then there exists $\lambda>1$ so that $u^{\lambda}(x)=\lambda^{\frac{d}{2}}u(\lambda x)$ obeys $\Phi(u^{\lambda})=0$ and $E(u^{\lambda})<E(u)$. Note that $M(u^{\lambda})=M(u)$ and $\int_{\R^d}|\nabla u^{\lambda}|^2\,dx=\lambda^2\int_{\R^d}|\nabla u|^2\,dx>\int_{\R^d}|\nabla u|^2\,dx$.
\end{lemma}
\begin{proof}
By direct computation, we have
\begin{equation}
\label{eq900}
\Phi(u^{\lambda})=\frac{\lambda}{2}\frac{d E(u^{\lambda})}{d\lambda}=\lambda^2\int_{\R^d}|\nabla u|^2\,dx-\frac{dp_1\lambda^{\frac{dp_1}{2}}}{2(p_1+2)}\int_{\R^d}|u|^{p_1+2}\,dx+\frac{dp_0\lambda^{\frac{dp_0}{2}}}{2(p_0+2)}\int_{\R^d}|u|^{p_0+2}\,dx.
\end{equation}
Suppose $\Phi(u)<0$. By \eqref{eq900}, $\lim_{\lambda\rightarrow \infty} \Phi(u^{\lambda})=\infty$. Thus, there exists $\lambda>1$ such that $\Phi(u^{\lambda})=0$. Let $\lambda_0$ be smallest such $\lambda$. It implies that from \eqref{eq900}, $E(u^{\lambda})$ is decreasing on $[1,\lambda_0)$. Thus, $E(u^{\lambda_0})<E(u)$.
\end{proof}

\begin{lemma}(\cite[Lemma 5.4]{KiOhPoVi17})
\label{Lm:exist desired profile}
Assume $d=1,2$ or $d\geq 3$ and $p_0\leq\frac{4}{d-2}$. Let $m>0$ and $u \in H^1(\R^d)$ be such that $0<M(u)<m$, $\Phi(u)=0$. Then there exists $v\in H^1(\R^d)$ satisfying: $M(v)=m$, $E(v)< E(u)$ and $\Phi(v)=0$.  
\end{lemma}
\begin{proof}
The proof is similar as in \cite{KiOhPoVi17}[Lemma 5.4]. Remark that $H^1 \hookrightarrow L^{p_0+2}$ since $p_0\leq \frac{4}{d-2}$.
\end{proof}

We divide the proof of Theorem \ref{thm90} into three steps.

\subsubsection{Finiteness}
We argue as in \cite[Proof of Theorem 5.2]{KiOhPoVi17}. Assume $u\in H^1(\R^d)\cap L^{p_0+2}(\R^d)$ such that $(M(u),E(u))\in\mathcal{R}$. If $\Phi(u)=0$ then by the definition of the function $I^{\Phi}(\lambda)$, $I^{\Phi}(M(u))<E(u)$. If $\Phi(u)< 0$ then using Lemma \ref{lm90}, there exists $u^{\lambda}$ such that $\Phi(u^{\lambda})=0$ and $E(u^{\lambda})<E(u)$. Thus, $I^{\Phi}(M(u))=I^{\Phi}(M(u^{\lambda}))\leq E(u^{\lambda})<E(u)$. In all cases, this gives a contradiction with the definition of $\mathcal{R}$. Thus, $\Phi(u)>0$.

If $M(u)=m<\tilde{m}_c$ then by Theorem \ref{thm1}, $\Phi(u)>0$. Thus $I^{\Phi}(m)=\infty$.

If $\tilde{m}_c\leq m<m_c$. By Claim \ref{Cl:nabla} and Theorem \ref{thm1}, $I^{\Phi}(m)>0$. Let $u \in H^1(\R^d)\cap L^{p_0+2}(\R^d)$ be such that $M(u)=m$. By Corollary \ref{cor2}, there exists $u_0\in H^1(\R^d)\cap L^{p_0+2}(\R^d)$ such that $M(u_0)=m$ and $\Phi(u_0)=\tilde{I}(m)\leq 0$. We divide into two cases. If $m=\tilde{m}_c$ then $\Phi(u_0)=\tilde{I}(m)=0$. This implies that $0<I^{\phi}(m)\leq E(u_0)<\infty$. If $\tilde{m}_c<m<m_c$ then $\Phi(u_0)=\tilde{I}(m)<0$. By Lemma \ref{lm90}, there exists $\lambda_0>1$ such that $\Phi(u_0^{\lambda_0})=0$, $E(u_0^{\lambda_0})<E(u_0)$, $M(u_0^{\lambda_0})=m$. Combining all the above, we have
\[
0<I^{\Phi}(m)\leq E(u_0^{\lambda_0})<E(u_0)<\infty.
\]

If $m\geq m_c$ then by Theorem \ref{thm1}, there exists $u_0\in H^1(\R^d)\cap L^{p_0+2}(\R^d)$ such that $M(u_0)=m$, $E(u_0)=I(m)\leq 0$. Since $u_0$ is a minimizer of problem $I(m)$, from Theorem \ref{thm1}, we have $\Phi(u_0)=0$.
Thus, $I(m)=I^{\Phi}(m)$ and $I^{\Phi}(m)$ is achieved. Moreover, using Theorem \ref{thm1}, $I^{\Phi}(m)$ is negative, strictly decreasing on $(m_c,\infty)$.

\subsubsection{Monotonicity}
\label{ssub:monotonicity}

We next prove that $m\mapsto I^{\phi}(m)$ is nonincreasing on $[\tilde{m}_c,m_c]$. Let $m\in [\tilde{m}_c,m_c]$ and $m'>m$ such that $m'\leq m_c$. We will prove that for all $\eps>0$, there exists $U_{\eps}\in H^1\cap L^{p_0+2}$ such that $M(U_{\eps})=m'$, $\Phi(U_{\eps})\leq 0$ and $E(U_{\eps})\leq I^{\phi}(m)+\eps$. Using Lemma \ref{lm90}, we see that it would impliy $I^{\phi}(m')\leq I^{\phi}(m)+\eps$, and thus, letting $\eps$ goes to $0$, $I^{\phi}(m')\leq I^{\phi}(m)$, which yields that $I^{\phi}$ is nonincreasing.

To prove the existence of $U_{\eps}$ satisfying the desired property, we first consider a function $v\in H^1\cap L^{p_0+2}$, given by the definition of $I^{\phi}$, such that
\begin{equation}
 \label{Pv9}
 M(v)=m,\ \Phi(v)=0\text{ and } E(v)\leq I^{\phi}(m)+\frac{\eps}{8}.
\end{equation}
We let $v_b(x)=v(bx)$, for some $b<1$. Then
\begin{equation}
 \label{Pv10}
 \int |v_b|^2=b^{-d}\int |v|^2=m+o(1)\text{ as }b\to 1.
 \end{equation}
 Furthermore
$\Phi(v_b)= b^{-d}\left(b^2\int_{\R^d}|\nabla v|^2\,dx-\frac{dp_1}{2(p_1+2)}\int_{\R^d}|v|^{p_1+2}\,dx+\frac{dp_0}{2(p_0+2)}\int_{\R^d}|v|^{p_0+2}\,dx\right)$
which implies, using that $\Phi(v)=0$,
 \begin{equation}
 \label{Pv11}
 \Phi(v_b)=b^{-d}(b^2-1)\int |\nabla v|^2<0,\text{ when }b<1.
\end{equation}
We also have
\begin{equation}
\label{Pv12}
E(v_b)=E(v)+o(1)\text{ as }b\to 1.
\end{equation}
By \eqref{Pv9}, \eqref{Pv10} and \eqref{Pv12} one can choose $b<1$ such that
\begin{equation}
 \label{Pv13}
 M(v_b)\leq m',\quad \Phi(v_b)<0,\quad E(v_b)\leq I^{\phi}(m)+\frac{\eps}{4}.
\end{equation}
If $M(v_b)=m'$, we can take $U_{\eps}=v_b$ and we are done.

If $M(v_b)<m'$, we fix $\psi\in C_0^{\infty}(\R^d\setminus \{0\})$ and $M(\psi)=m'-M(v_b)$, and we let
$$\tilde{u}_{\lambda}(x)=v_b(x)+\lambda^{d/2}\psi(\lambda x),\quad u_{\lambda}=\frac{\sqrt{m'}}{\|\tilde{u}_{\lambda}\|_2}\tilde{u}_{\lambda}.$$
Then
\begin{gather}
 \label{Pv20}
 \lim_{\lambda\to 0} \|\tilde{u}_{\lambda}\|_{2}^2=M(v_b)+M(\psi)=m'\\
 \label{Pv21}
 \|u_{\lambda}\|^2_{2}=m'.
\end{gather}
Furthermore, since $\frac{4}{d}<p_1<p_0$,
$$\sum_{j=0,1}\lim_{\lambda \to 0}\|\tilde{u}_{\lambda}-v_b\|_{L^{p_j+2}}+\|\nabla(\tilde{u}_{\lambda}-v_b)\|_{L^2}=0,$$
and thus, by \eqref{Pv20} and the definition of $u_{\lambda}$,
$$\sum_{j=0,1}\lim_{\lambda \to 0}\|u_{\lambda}-v_b\|_{L^{p_j+2}}+\|\nabla(u_{\lambda}-v_b)\|_{L^2}=0,$$
which implies, using also \eqref{Pv13},
\begin{equation}
 \label{Pv22}
 \lim_{\lambda\to 0}\Phi(u_{\lambda})=\Phi(v_b)<0,\quad \lim_{\lambda\to 0}E(u_{\lambda})=E(v_b)\leq I^{\Phi}(m)+\frac{\eps}{2}.
\end{equation}
By \eqref{Pv22}, we can choose $\lambda$ large such that
$$\Phi(u_{\lambda})\leq 0,\quad E(u_{\lambda})\leq I^{\phi}(m)+\eps.$$
Combining with \eqref{Pv21} we see that $U_{\eps}=u_{\lambda}$ satisfies the desired properties.

\subsubsection{Strict monotonicity}


Now, we prove the last statement in Theorem \ref{thm90}. From Theorem \ref{thm1} and $I^{\Phi}(m)=I(m)$ for $m\geq m_c$, we see that $I^{\Phi}$ is strictly decreasing and achieved on $[m_c,\infty)$. Now, consider $I^{\Phi}$ on $[\tilde{m}_c,m_c]$. Assume that $d=1,2,3,4$ or $d\geq 5$ and $p_0\leq\frac{4}{d-2}$.



We start by proving that $I^{\phi}(m)$ is achieved. Note that it is already known in the case $m=m_c$. We thus assume $ \tilde{m}\leq m<m_c$. We consider a minimizing sequence $u_n$ for $I^{\phi}(m)$:
$$ \Phi(u_n)=0, \quad E(u_n)\leq I^{\phi}(m)+\frac{1}{2^{n}}, \quad \int |u_n|^2=m.$$
Let $v_n$ be the symmetric decreasing rearrangement of $u_n$. Then
\begin{equation}
 \label{bound_vn}
\Phi(v_n)\leq 0,\quad E(v_n)\leq E(u_n)\leq I^{\phi}(m)+\frac{1}{2^n},\quad \int |v_n|^2=m.
 \end{equation}
Let
$$a=\max_{s\geq 0} \left( \frac{p_1 d}{2(p_1+2)}s^{p_1}-\frac{p_0 d}{2(p_0+2)}s^{p_0}\right)\in (0,\infty).$$
Then
$$-\frac{p_1d}{2(p_1+2)}\int |v_n|^{p_1+2}+\frac{p_0d}{2(p_0+2)}\int |v_n|^{p_0+2}+aM(v_n)\geq 0,$$
and $\Phi(v_n)\leq 0$, $\int |v_n|^2=m$ imply $ \int |\nabla v_n|^2\leq am.$
Thus $v_n$ is bounded in $\dot{H}^1$. Since $\Phi(v_n)\leq 0$, using Gagliardo-Nirenberg inequality and the boundedness in $L^2$, we deduce that $v_n$ is bounded in $H^1\cap L^{p_0+2}$.  Let $v\in H^1\cap L^{p_0+2}$ such that (after extraction)
$$v_n \xrightharpoonup[]{n\to\infty}v \text{ weakly in }H^1\cap L^{p_0+2}.$$
Since $v_n$ is radial and a nonincreasing function of $|x|$, using Strauss Lemma, we obtain that the convergence is strong in $L^q$, $2<q<\frac{2d}{d-2}$ (for all $2<q<\infty$ if $d=1,2$). Also, since $(v_n)_n$ is bounded in $L^{p_0+2}$, the convergence is strong in $L^q$ for all $2<q<p_0+2$. In particular, it is strong in $L^{p_1+2}$. We thus obtain:
\begin{gather*}
\Phi(v)\leq \liminf_{n}\Phi(v_n)\leq 0\\
E(v)\leq \liminf_{n}E(v_n)\leq I^{\phi}(m)\\
 M(v)\leq \liminf_{n}M(v_n)\leq m.
\end{gather*}
By Lemma \ref{lm90} and the fact that $\Phi(v)$ is nonpositive, we obtain $E(v)\geq I^{\phi}(M(v))$.
Since by \S \ref{ssub:monotonicity}, $I^{\phi}$ is nonincreasing, we obtain
$$ I^{\phi}(m)\leq I^{\phi}(M(v))\leq E(v)\leq I^{\phi}(m),$$
which proves that all these quantities must be equal. Thus $\lim_{n}E(v_n)=E(v),$
which shows that $v_n$ converges strongly to $v$ in $\dot{H^1}$ and in $L^{p_0+2}$. Thus we have found $v\in L^{p_0+2}\cap H^1$ such that $\Phi(v)\leq 0$ and
\begin{equation}
\label{good_v1}
E(v)=I^{\phi}(m)=I^{\phi}(M(v)),\quad M(v)\leq m.
\end{equation}
Note that since $m<m_c$, we have $E(v)>0$ and thus $v$ is not identically $0$.
By Lemma \ref{lm90}, we see that we cannot have $\Phi(v)<0$. Thus
\begin{equation}
 \label{good_v2}
 \Phi(v)=0
\end{equation}
If $M(v)=m$ we are done. To conclude the proof that $I^{\phi}(m)$ is attained, we will assume $M(v)<m$ and obtain a contradiction.

First assume $d\in\{1,2,3,4\}$. Denoting $m'=M(v)$ , we see by the above that $I^{\phi}(m)=I^{\phi}(m')$ and since $I^{\phi}$ is nonincreasing, we deduce
$$ I^{\phi}(\mu)=I^{\phi}(m),\quad m'\leq \mu\leq m.$$
By the definition of $I^{\phi}$ we obtain
\begin{equation}
\label{minime_v}
E(v)=\min_{\substack{0<M(u)\leq m\\ \Phi(u)=0}}E(u).
\end{equation}
Since $M(v)=m'<m$, we see that $v$ is a minimum of $E(u)$ with the sole constraint $\Phi(u)=0$. Thus there exists a Lagrange multiplier $\alpha$ such that
\begin{equation}
\label{equation_v0}
-\Delta v-v^{p_1+1}+v^{p_0+1}=\alpha\left(-\Delta v-\frac{dp_1}{4}v^{p_1+1}+\frac{dp_0}{4}v^{p_0+1}\right).
\end{equation}
If $\alpha=1$, we obtain
$$\left( \frac{dp_1}{4}-1 \right)v^{p_1+1}=\left( \frac{dp_0}{4}-1 \right)v^{p_0+1},$$
which shows, since $v\in H^1$, that $v=0$, a contradiction.

If $\alpha\neq 1$, we see that $v$ is a $H^1$ solution of an equation of the form
$$-\Delta v+a_1v^{p_1+1}+a_0v^{p_0+1}=0,$$
where $(a_0,a_1)\in \R^2$. Since we have assumed $d\in \{1,2,3,4\}$, this proves again by Proposition \ref{P:uniqueness} that $v=0$, a contradiction which concludes the proof if $d\in \{1,2,3,4\}$.\\

We next assume that $d\geq 5$ and $p_0\leq \frac{4}{d-2}$.
By Lemma \ref{Lm:exist desired profile}, there exists $\tilde{v}\neq 0$ such that $M(\tilde{v})=m$, $\Phi(\tilde{v})=0$ and $E(\tilde{v})<E(v)=I^{\Phi}(m)$, this gives a contradiction to the definition of $I^{\Phi}(m)$. Thus, the achievement of $I^{\Phi}(m)$ is proved.\\

Now, we prove $I^{\Phi}(m)$ is strictly decreasing on $[\tilde{m}_c,m_c]$. Let $\tilde{m}_c\leq m_1<m_2\leq m_c$. Since $I^{\Phi}(m_1)$ is achieved, let $u\in H^1(\R^d)\cap L^{p_0+2}(\R^d)$ be such that $M(u)=m_1$, $\Phi(u)=0$, $E(u)=I^{\Phi}(m_1)$. Assume that $p_0\leq \frac{4}{d-2}$. By Lemma \ref{Lm:exist desired profile}, there exists $\tilde{u}$ such that $M(\tilde{u})=m_2$, $\Phi(\tilde{u})=0$, $E(\tilde{u})<E(u)=I^{\Phi}(m_1)$. Thus, by the definition of $I^{\Phi}(m_2)$, we have $I^{\Phi}(m_1)>I^{\Phi}(m_2)$, which proves that $I^{\Phi}$ is strictly decreasing on $[\tilde{m}_c,m_c]$ in the case $p_0\leq \frac{4}{d-2}$. We next assume that $d\in\{1,2,3,4\}$ and $I^{\phi}(m_1)=I^{\phi}(m_2)$. Since $I^{\phi}(m)$ is nonincreasing on $[\tilde{m}_c,\infty)$, we have 
\[
E(u)=\min_{\substack{0<M(v)\leq m_2\\ \Phi(v)=0}}E(v).
\]
Similarly as above, using Proposition \ref{P:uniqueness}, we have $u=0$, a contradiction which implies the desired result.\\

Next, we prove that $I^{\Phi}$ is lower semicontinuous on $[\tilde{m}_c,\infty)$. Since, $I^{\Phi}$ is nonincreasing, it suffices to show that $I^{\Phi}$ is right continuous. We argue by contradiction, assume that there exists $m_n$ decreasing to $m$ with $I^{\Phi}(m)>\lim_{n\rightarrow\infty}I^{\Phi}(m_n)$. Let $u_n$ be radial functions such that $M(u_n)=m_n$, $\Phi(u_n)=0$, $E(u_n)=I^{\Phi}(m_n)$. We see that $u_n$ is bounded in $H^1(\R^d)$. Thus, $u_n$ weakly converges to $v$ in $H^1(\R^d)$ and strongly converges to $v$ in $L^{p_1+2}(\R^d)$. This implies that $M(v)\leq m$, $\Phi(v)\leq 0$, $E(v)\leq\lim_{n\rightarrow\infty}E(u_n)<I^{\Phi}(m)$. If $\Phi(v)=0$ then $I^{\Phi}(m)\leq I^{\Phi}(M(v))\leq E(v)<I^{\Phi}(m)$, which gives a contradiction. If $\Phi(v)<0$ then by Lemma \ref{lm90}, there exists $\tilde{v}$ such that $M(\tilde{v})=M(v)\leq m$, $\Phi(\tilde{v})=0$, $E(\tilde{v})<E(v)<I^{\Phi}(m)$. Similar as the above, we gives a contradiction. This completes the proof.


\begin{proof}[Proof of Theorem \ref{thm3}]
For $m\in (0,m_c)$, let $\mathsf{e}(m)=I^{\Phi}(m)$. Theorem \ref{thm90} implies that $\mathsf{e}(m)$ is nonincreasing on $(0,m_c)$. For each $u \in H^1(\R^d)\cap L^{p_0+2}(\R^d)$ such that $M(u)=m\in (0,m_c)$ and $0<E(u)<\mathsf{e}(m)$, let $\xi=\frac{P(u)}{M(u)}$ and $u_{\xi}(x)=e^{ix\xi}u(x)$. We have $M(u_{\xi})=M(u) \in (0,m_c)$ and $0<E(u_{\xi})=E(u)-\frac{|P(u)|^2}{M(u)}\leq E(u)<\mathsf{e}(m)\leq I^{\Phi}(m)$. Using Theorem \ref{thm90}, $\Phi(u_{\xi})> 0$. This implies that
\[
\Phi(u)-\frac{|P(u)|^2}{M(u)}> 0.
\]
This proves the desired result.
\end{proof}


\section{Scattering}
 \label{S:scattering}

In this subsection, we conclude the proof of Theorem \ref{T:scatt_intro2}. 

\subsection{Proof of Theorem \ref{T:scatt_intro2} in the energy-supercritical case}
We prove here Theorem \ref{T:scatt_intro2} in the case where $s_0>1$.

We consider the set
$$\RRR:= \left\{\varphi\in H^{s_0},\; 0< \int|\varphi|^2<m_c,\; E(\varphi)<\mathsf{e}\Big(\int |\varphi|^2\Big)\right\},$$ 
where $m_c$, $\mathsf{e}$ are defined in Section \ref{S:var}.


We will prove Theorem \ref{T:scatt_intro2} as a consequence of 
\begin{theorem}
 \label{T:scatt2}
 Suppose that Assumption \ref{Assum:NL} holds and that $s_0>1$.  Let $A_0$ such that Property \ref{Proper:bnd} holds.
 For $\eta>0$, we denote 
 $$\RRR_{\eta}=\left\{\varphi\in H^{s_0}, 0< \int |\varphi|^2\leq m_c-\eta,\; E(\varphi)\leq R(\varphi)-\eta\right\},$$
 where $R(\varphi)=\mathsf{e}\Big(\int |\varphi|^2\Big)$. Then for all $A\in (0,A_0)$ there exists $\FFF(A,\eta)>0$ such that for any interval $I$, for any solution $u\in C^0(I,H^{s_0})$ of \eqref{NLS} such that
 \begin{equation}
  \label{sct20}
 \exists t\in I,\; u(t)\in \RRR_{\eta} \quad
 \text{and}\quad\sup_{t\in I} \|u(t)\|_{\dot{H}^{s_0}}^2+\eta\|u(t)\|^2_2\leq A^2,
 \end{equation} 
 one has $u\in S^{s_0}(I)$ and $\|u\|_{S^{s_0}(I)}\leq \FFF(A,\eta)$.
\end{theorem}
Theorem \ref{T:scatt2} implies Theorem \ref{T:scatt_intro2} by the scattering criterion of Theorem \ref{T:local wellposed}, the conservation of the $L^2$ norm and the fact that $\RRR=\cup_{\eta>0} \RRR_{\eta}$.
\begin{proof}[Proof of Theorem \ref{T:scatt2}]
 We argue by contradiction, following the compactness/rigidity scheme as in \cite{KeMe06}. We fix $\eta>0$ throughout the argument. In all the proof, we will endow $H^{s_0}$ with the norm defined by
 \begin{equation}
  \label{Hs0eta}
  \|u\|^2_{H^{s_0}}=\|u\|^2_{\dot{H}^{s_0}}+\eta\|u\|^2_2.
 \end{equation}
 We will denote by $\PPP(A)$ the property that there exists $\FFF(A)$ such that for any interval $I$, for any solution $u\in C^0(I,H^{s_0})$ such that \eqref{sct20} holds, one has $u\in S^{s_0}$ and $\|u\|_{S^{s_0}(I)}\leq \FFF(A)$.
 
 By the small data theory for \eqref{NLS}, if $A>0$ is small enough and $\|u(t)\|_{H^{s_0}}\leq A$ for some $t\in I_{\max}(u)$, then $u$ is globally defined, scatters and $\|u\|_{S^{s_0}(\R)}\lesssim A$. This implies that $\PPP(A)$ holds for small $A>0$.

 Thus if the conclusion of Theorem \ref{T:scatt2} does not hold, there exists $A_c\in (0,A_0)$ such that for all $A<A_c$, $\PPP(A)$ holds, and $\PPP(A_c)$ does not hold, i.e. there exists a sequence of intervals $((a_n,b_n))_n$, a sequence $(u_n)_n$ of solutions of \eqref{NLS} on $(a_n,b_n)$,
 \begin{equation}
  \label{sct32}
  u_n\in C^0((a_n,b_n),H^{s_0}),\quad \exists t\in I_n,\; u_n(t)\in \RRR_{\eta},\quad 
  \lim_{n\to\infty}\sup_{a_n<t<b_n} \|u_n(t)\|_{H^{s_0}} =A_c.
 \end{equation} 
 and $\lim_{n\to\infty} \|u_n\|_{S^{s_0}((a_n,b_n))}=\infty$. Time translating $u_n$, we can assume
 \begin{equation}
  \label{sct33}
a_n<0<b_n,\quad 
  \lim_{n\to\infty}\|u_n\|_{S^{s_0}((a_n,0))}=\lim_{n\to\infty}\|u_n\|_{S^{s_0}((0,b_n))}=+\infty.
 \end{equation} 
 We will prove
 \begin{claim}
 \label{Cl:compactness}
  For any sequences $(a_n)_n$, $(b_n)_n$ with $a_n<0<b_n$, for any sequence $(u_n)_n$ of solutions of \eqref{NLS} satisfying \eqref{sct32}, \eqref{sct33}, there exist, after extraction of subsequences, a sequence $(x_n)_n\in (\R^d)^{\mathbb{N}}$ and $\varphi\in H^{s_0}$ such that
  \begin{equation}
  \label{convergence}
  \lim_{n\to\infty}\|u_n(0,\cdot-x_n)-\varphi\|_{H^{s_0}}=0.
  \end{equation} 
 \end{claim}
We first assume the claim and conclude the proof of Theorem \ref{T:scatt2}. By the claim, there exist (after extraction of subsequences) $\varphi\in H^{s_0}$ and $(x_n)_n$ such that \eqref{convergence} holds. Let $u$ be the solution of \eqref{NLS} such that $u(0)=\varphi$. Since $\RRR_{\eta}$ is closed in $H^{s_0}$, we have $\varphi\in \RRR_{\eta}$.

We next prove by contradiction
\begin{equation}
\label{lim_an_bn}
\lim_{n\to\infty}a_n=-\infty,\quad \lim_{n\to\infty} b_n=+\infty. 
\end{equation} 
Assume to fix ideas, and after extraction of subsequences $\lim_{n\to\infty} b_n =b\in [0,\infty)$. Using that $\lim_n\|u_n\|_{S^{s_0}([0,b_n))}=\infty$, we must have $T_{+}(u)<\infty$ and $b\geq T_{+}(u)$. By the last assertion of \eqref{sct32}, we obtain
$$\sup_{0\leq t<T_{+}(u)} \|u(t)\|_{H^{s_0}}\leq A_c.$$
This implies by Theorem \ref{T:GWP} that $T_{+}(u)=+\infty$, a contradiction. Hence \eqref{lim_an_bn}. Next, we see that \eqref{convergence}, perturbation theory for equation \eqref{NLS} and the last assertion in \eqref{sct32} implies that for any compact interval $I\subset I_{\max}(u)$,
$$
\sup_{t\in I}\|u(t)\|_{H^{s_0}}\leq A_c.$$
This implies by Theorem \ref{T:GWP} that $u$ is global and 
$$\sup_{t\in \R}\|u(t)\|_{H^{s_0}}\leq A_c.$$
By \eqref{sct33} and stability theory for equation \eqref{NLS}, one has
$$ \|u\|_{S^{s_0}((-\infty,0))}=\|u\|_{S^{s_0}((0,+\infty))}=+\infty.$$

If $(t_n)_n$ is any sequence of times, Claim \ref{Cl:compactness} and the preceding properties imply that one can extract subsequence such that $u(t_n,\cdot-x_n)$ converges in $H^{s_0}$ for some sequence $(x_n)_n\in (\R^d)^{\mathbb{N}}$. This is classical that it implies that one can find a function $x(t)$, $t\in \R$ such that $K$ defined by \eqref{defK} has compact closure in $H^{s_0}$. We give a sketch of proof of this fact. Using the compactness and the fact that the solution $u$ is not identically $0$, we first notice that there exists $R_0>0$, $\eta>0$ such that
$$\inf_{t\in \R}\left(\sup_{X\in \R^d}\int_{|x|<R_0}|u(t,x+X)|^2dx\right)\geq \eta.$$
Thus for all $t$, there exists $x(t)\in \R^d$ such that
$$ \int_{|x|<R_0} |u(t,x+x(t))|^2dx \geq \eta/2.$$
For this choice of $x(t)$, one can check that $K$ defined by \eqref{defK} is compact.

By Proposition \ref{pro1}, we have 
\begin{equation}
\label{Eq:bounded of Phi}
\min_{t\geq 0}\left|\Phi(u(t))-\frac{|P(u)|^2}{M(u)}\right|=0.
\end{equation}
By \eqref{Eq:bounded of Phi}, there exists a sequence of times $(t_n)_n$ such that
\begin{equation} 
\label{Eq:convergences}
\lim_{t_n\rightarrow\infty}\Phi(u(t_n))-\frac{|P(u(t_n))|^2}{M(u(t_n))}=0.
\end{equation}
By the claim, extracting subsequences, there exists $(x_n)_n$ such that $u(t_n,\cdot-x_n)$ convergences to $\varphi_0$ in $H^{s_0}$ (up to extract subsequence). By \eqref{Eq:convergences},
\begin{equation}
\label{Eq:varphi0}
\Phi(\varphi_0)-\frac{|P(\varphi_0)|^2}{M(\varphi_0)}=0.
\end{equation}
Since $\RRR_{\eta}$ is closed in $H^{s_0}$, $\varphi_0\in \RRR_{\eta}$. This implies, by Theorem \ref{thm3}, that
$$ \Phi(\varphi_0)-\frac{|P(\varphi_0)|^2}{M(\varphi_0)}>0.$$
This contradicts to \eqref{Eq:varphi0}. This completes the proof.\\



\end{proof}

We are left with proving Claim \ref{Cl:compactness}.
\begin{proof}[Proof of Claim \ref{Cl:compactness}]
 
 \emph{Step 1. Profile decomposition.}
 Extracting subsequences, we can assume that $u_{0,n}=u_n(0)$ has a profile decomposition as in Subsection \ref{sub:profile}:
 $$u_{0,n}=\sum_{j=1}^J \varphi^j_{Ln}(0)+w_n^J(0),$$
 where $\varphi^j_{Ln}=\frac{1}{\lambda^j_n}\varphi^j_L\left( \frac{t-t_{n}^j}{\lambda_{n}^j},\frac{x-x_{n}^j}{\lambda^j_n} \right)$.
 We denote by $\varphi^j_{n}$ the corresponding nonlinear profiles, and $\tilde{\varphi}^j_n$ the modified nonlinear profiles. 
 
 Our goal is to prove that there is a unique $j_0\geq 1$ such that $\varphi^{j_0}$ is not identically zero, and that $j_0\in \JJJ_{NC}$. We first note that there is at least one $j$ such that $\varphi^j$ is not identically zero. If not, by the small data local well-posedness, we would have 
 $$\lim_{n\to\infty}\|u_n\|_{S^{s_0}(I_n)}=0,$$
 a contradiction with our assumptions.
 
 In the remaining step, we will prove that there is at most one nonzero profile. Arguing by contradiction, we assume that there is at least two nonzero profiles, say $\varphi^1$ and $\varphi^2$. By the small data theory, there exists $\eps_0>0$ such that 
 \begin{equation}
  \label{sct40}
\inf_{t\in I_n}\left\|\tilde{\varphi}^j_n\right\|_{H^{s_0}}\geq \eps_0,\quad j\in \{1,2\}. 
  \end{equation} 
  
\emph{Step 2. Bound of the $H^{s_0}$ norm.}
  We prove that for all $j\geq 1$ we have $I_n\subset I_{\max}(\tilde{\varphi}^j_n)$ and, for large $n$, 
  \begin{equation}
   \label{sct41}
   \sup_{t\in I_n} \|\tilde{\varphi}_n^j(t)\|_{H^{s_0}} \leq \sqrt{A_c^2-\frac 14\eps_0^2}.
  \end{equation} 
  By \eqref{sct32}, and the Pythagorean expansions \eqref{Pythagorean1}, \eqref{PythagoreanHs}, we obtain the bounds, for $J\geq 1$
   \begin{equation*}
    \sum_{j\geq 1} \|\varphi^j_L(0)\|^2_{\dot{H}^{s_0}}\leq A_c^2,\quad \sum_{j\in \JJJ_{NC}} \|\varphi^j_L(0)\|^2_{H^{s_0}}\leq A_c^2
   \end{equation*}
   Fixing a small $\eps>0$, we obtain, by the small data theory for equations \eqref{NLS} and \eqref{NLSh} and Lemma \ref{L:modifprofile1}, that there exists $J_0\geq 1$ such that, for $j\geq J_0+1$, $I_{\max}(\varphi^j)=\R$,
\begin{equation}
\label{boundSlargej}
\forall j\geq J_0+1,\quad
\begin{cases}
   \|\varphi^j\|_{\dot{S}^{s_0}(\R)}<\infty &\text{ if }j\in \JJJ_C\\
   \|\varphi^j\|_{S^{s_0}(\R)}<\infty &\text{ if }j\in \JJJ_{NC}.
  \end{cases}
 \end{equation}    
and
\begin{equation}
\label{boundHs2_1}
\forall j\geq J_0+1,\quad
\limsup_{n\to\infty}\sup_{t\in \R}\|\tilde{\varphi}_n^j(t)\|_{H^{s_0}}\leq \eps.
\end{equation} 
We next prove by contradiction that for $j\in \llbracket 1, J_0\rrbracket$, $I_n\subset I_{\max}(\tilde{\varphi}_n^j)$ and  \eqref{sct41} holds. If not, we can assume (inverting time if necessary, and using Theorem \ref{T:GWP}) that for large $n$, there exists $b_n'\in (0,b_n]$ such that $[0,b_n']\subset \bigcap_{1\leq j\leq J_0} I_{\max}(\tilde{\varphi}_n^j)$ and
\begin{equation}
\label{boundHs2_absurd}
\sup_{1\leq j\leq J_0}\sup_{0\leq t\leq b_n'}\left\|\tilde{\varphi}_n^j(t)\right\|_{H^{s_0}}^2=A_c^2-\frac{1}{2}\eps_0^2. 
\end{equation} 
This implies  that for large $n$
$$ \forall j\in \JJJ_{C}\cap \llbracket 1,J_0\rrbracket,\quad  \sup_{0\leq t\in b_n'}\|\varphi_n^j(t)\|_{\dot{H}^{s_0}}< A_c<A_0.$$
Thus by Proposition \ref{P:boundh}, we obtain a constant $C>0$ such that for large $n$,
\begin{equation}
\label{boundJC}
\sup_{j\in \JJJ_C\cap \llbracket 1,J_0\rrbracket}\|\varphi^j_n\|_{\dot{S}^{s_0}([0,b_n'])}\leq C. 
\end{equation} 
Going back to \eqref{boundHs2_absurd}, we see also that for large $n$
$$\forall j\in \llbracket 1,J_0\rrbracket \cap \JJJ_{NC},\quad \sup_{0\leq t\leq b_n'} \|\varphi_n^j(t)\|_{H^{s_0}}\leq \sqrt{A_c^2-\frac{1}{4}\eps_0^2}.$$

Also, using the Pythagorean expansion of the mass we see that 
$$\forall j\in \JJJ_{NC}, \; M(\varphi^j)\leq m_c-\eta.$$
This implies that $E(\varphi^j)\geq 0$ for $j\in \JJJ_{NC}$. By the Pythagorean expansion \eqref{PythagoreanHs} of the $\dot{H}^1$ norm and Remark \ref{R:Lq}, we obtain
$$\sum_{j\in \JJJ_{NC}} E(\varphi^j)\leq \liminf_n E(u_{0,n})\leq \liminf_n e \left(\int |u_{0,n}|^2\right)-\eta$$
and thus 
$$ \forall j\in \JJJ_{NC}, \quad E(\varphi^j)\leq R(\varphi^j)-\eta$$
where we used the fact that $\mathsf{e}$ is nonincreasing. Thus $\varphi^j_n(0)\in \RRR_{\eta}$. Using that $\PPP(A)$ holds for $A=\sqrt{A_c^2-\frac{1}{4}\eps_0^2}$ we obtain that there exists a constant $C>0$ such that for large $n$
\begin{equation}
\label{boundJNC}
\sup_{j\in \JJJ_{NC}\cap \llbracket 1,J_0\rrbracket}\|\varphi^j_n\|_{S^{s_0}([0,b_n'])}\leq C. 
\end{equation} 
Combining \eqref{boundSlargej}, \eqref{boundJC} and \eqref{boundJNC}, we obtain that the assumptions of Theorem \ref{T:NLapprox} are satisfied on $[0,b_n']$. Using the Pythagorean expansion of Lemma \ref{L:Pythagorean} together with the limit in \eqref{sct32}, we obtain
$$\limsup_{n\to\infty} \sup_{0\leq t\leq b_n'} \sum_{j=1}^{J_0}\left\|\tilde{\varphi}_n^j(t)\right\|_{H^{s_0}}^2\leq A_c^2.$$
By \eqref{sct40}, we deduce
$$ \forall j\in \llbracket 1,J_0\rrbracket, \quad \limsup_{n\to\infty}\sup_{0\leq t\leq b_n'} \left\|\tilde{\varphi}_n^j(t)\right\|_{H^{s_0}}^2\leq A_c^2-\eps_0^2,$$
contradicting \eqref{boundHs2_absurd}. This proves that \eqref{sct41} holds for all $j\geq 1$, for large $n$.

\medskip

\emph{Step 3. Uniqueness of the nonzero profile.}

In this step we still assume that $\varphi^1$ and $\varphi^2$ are nonzero profiles.
Using \eqref{sct32} and \eqref{sct41}, and arguing as in Step 2, we see that the assumptions of Theorem \ref{T:NLapprox} are satisfied on $[a_n,b_n]$. This proves that $u_n$ scatters for large $n$, contradicting \eqref{sct33}. This concludes the proof that there is only one nonzero profile.

\medskip

\emph{Step 4. End of the proof.}

We assume that $\varphi^1$ is the only nonzero profile. By the same argument as before, we obtain that for large $n$, $I_n\subset I_{\max}(\tilde{\varphi}^1_n)$ and
$$\lim_{n\to\infty}\sup_{t\in I_n} \|\tilde{\varphi}^1_n\|_{H^{s_0}}\leq A_c.$$
If $1\in \JJJ_{C}$, we obtain by Proposition \ref{P:boundh} that $\limsup_{n\to\infty}\|\varphi^1_n\|_{\dot{S}^{s_0}(I_n)}<\infty$. Thus the assumptions of Theorem \ref{T:NLapprox} are satisfied on $I_n$, a contradiction with \eqref{sct33}. Thus $1\in \JJJ_{NC}$. By the same argument, we obtain $\limsup_{n\to\infty} \sup_{t\in I_n} \|\tilde{w}_{Ln}^1(t)\|_{H^{s_0}}=0$. Indeed, if not, we would have by the conservation of the $H^{s_0}$ norm for the linear Schr\"odinger equation (and after extraction of a subsequence) $\lim_{n\to\infty} \sup_{t\in I_n} \|\tilde{w}_{Ln}^1(t)\|_{H^{s_0}}=\eps_0>0$, and the same strategy as in Steps 2,3 would yield that $u_n$ scatters for large $n$, a contradiction. We have proved
$$u_n(0,x)=\varphi^1_L(-t_{1,n},x-x_{1,n})+o(1)\text{ in }H^{s_0}.$$
By \eqref{sct33}, $t_{1,n}$ must be bounded, and we can assume $t_{1,n}=0$ for all $n$, i.e.
\begin{equation}
\label{simple_expansion}
u_n(0,x)=\varphi^1(0,x-x_{1,n})+o(1)\text{ in }H^{s_0},
\end{equation} 
which concludes the proof of the claim.
\end{proof}

\subsection{Sketch of proof of Theorem \ref{T:scatt_intro1}.}
The proof of Theorem \ref{T:scatt_intro1} in the case $s_0\leq 1$ is essentially the same Theorem \ref{T:scatt_intro2}, although a little simpler than the proof in the case $s_0>1$. One can prove the following analog of Theorem \ref{T:scatt2}:
\begin{theorem}
 \label{T:scatt3}
 Suppose that Assumption \ref{Assum:NL} holds and that $0<s_0\leq 1$.
For $\eta>0$, we denote 
 $$\RRR_{\eta}=\left\{\varphi\in H^{1}, 0< \int |\varphi|^2\leq m_c-\eta,\; E(\varphi)\leq R(\varphi)-\eta\right\}.$$
 Then for all $A>0$ there exists $\FFF(A,\eta)>0$ such that for any interval $I$, for any solution $u\in C^0(I,H^{1})$ of \eqref{NLS} such that 
 \begin{equation}
  \label{sct20'}
 \exists t\in I,\; u(t)\in \RRR_{\eta} \quad
 \text{and}\quad\sup_{t\in I} \|u(t)\|_{\dot{H}^{1}}^2+\eta\|u(t)\|^2_2\leq A^2,
 \end{equation} 
 one has
 $u\in S^{1}(I)$ and $\|u\|_{S^{1}(I)}\leq \FFF(A,\eta)$.
\end{theorem}
Note that in this case, by conservation of the energy, any solution of \eqref{NLS} is bounded in $H^1$, so that assumption \eqref{sct20'} is always satisfied for some $A$.

The proof of Theorem \ref{T:scatt3} goes along the same lines as the proof of Theorem \ref{T:scatt2}. Note however that if $s_0<1$, in the $H^{s_0}$ profile decomposition of $u_{0,n}$ in the proof of Claim \ref{Cl:compactness}, the set $\JJJ_c$ is empty since the sequence $(u_{0,n})$ is bounded in $H^1$.

\appendix

\section{Elliptic properties}
\begin{claim}[Pohozaev identity]
 \label{Cl:Pohozaev}
 Let $u\in H^1\cap L^{p_0+2}(\R^d)$, $0<p_1<p_0$. Assume that $u$ is radial or $u\in L^{\infty}_{\loc}(\R^d)$ and, for some real numbers $\alpha_1,\alpha_0,\mu$
 \begin{equation}
  \label{Poho:ell}
  -\Delta u +\alpha_1 |u|^{p_1}u+\alpha_0|u|^{p_0}u=\mu u.
 \end{equation}
 Let $c_j=\frac{dp_j}{2(p_j+2)}$, $j=0,1$. Then
 \begin{equation}
 \label{Pohozaev}
  \int |\nabla u|^2+c_1\alpha_1\int |u|^{p_1+2}+c_0\alpha_0\int |u|^{p_0+2}=0.
 \end{equation}
\end{claim}
\begin{proof}[Sketch of proof]
Multiply by $x\cdot \nabla \overline{u}+\frac{d}{2}\overline{u}$, take the real part and integrate by parts on $\R^d$. It is easy to check that the formal integration by parts are rigorous when $u\in L^{p_0+2}\cap H^1$ with $u$ radial or $L^{\infty}_{\loc}$.
See \cite{Po65} and \cite[Proposition 1]{BeLi83a}.
\end{proof}
We next prove a uniqueness result for the elliptic equation with double power nonlinearity:
\begin{proposition}
\label{P:uniqueness}
Let $\frac{4}{d}<p_1<p_0$, $d\in\{1,2,3,4\}$, $(a_0,a_1)\in \R^2$, and $u\in L^{p_0+2}\cap H^1$ be a radial solution of
\begin{equation}
 \label{Uniqueness:u}
 -\Delta u+a_1|u|^{p_1}u+a_0|u|^{p_0}u=0.
\end{equation}
Then $u\equiv 0$.
\end{proposition}
\begin{proof}
The case $a_0=a_1=0$ is trivial. We thus assume without loss of generality $a_1\neq 0$.

We have
\begin{equation}
 \label{eq:ell}
\partial_r(r^{d-1}\partial_r u)=r^{d-1}(a_0|u|^{p_0}u+a_1|u|^{p_1}u).
 \end{equation} 
By Strauss Lemma $|u(r)|\lesssim r^{\frac{1-d}{2}}\|u\|_{H^1}$, thus $u(r)$ is uniformly bounded for $r\geq 1$. By the equation \eqref{eq:ell}, using $1<p_1<p_0$,
$$|\partial_r(r^{d-1}\partial_r u)|\lesssim |u|^2r^{d-1}.$$
Since $\int |u|^2r^{d-1}dr<\infty$, the right-hand side of the preceding inequality is integrable, which proves that the following limit exists:
$$\lim_{r\to \infty}r^{d-1}\partial_ru=\ell.$$
If $\ell \neq 0$, we have $\int |\partial_ru|^2rdr=\infty$ (if $d=1,2$) or $\int |u|^{2}r^{d-1}dr=\infty$ (if $d=3,4$), contradicting our assumptions. Thus $\ell=0$, i.e.
\begin{equation}
 \label{goes_to_0}
 \lim_{r\to\infty}r^{d-1}\partial_ru(r)=0.
\end{equation} 

We first treat the case $d=1$. The equation \eqref{eq:ell} implies $F'(r)=0$, where
 $$ F(r)=\frac{1}{2}|u'(r)|^2-\frac{a_0}{p_0+2}|u(r)|^{p_0+2}-\frac{a_1}{p_1+2}|u(r)|^{p_1+2}.$$
Since $u\in H^1(\R)$, we have that $\int_{0}^{\infty} F(r)dr$ is finite, and thus that $F(r)=0$ for all $r>0$. Also, since $\lim_{r\to\infty}u'(r)=0$, we obtain that
$$\lim_{r\to\infty} \frac{a_0}{p_0+2}|u(r)|^{p_0+2}+\frac{a_1}{p_1+2}|u(r)|^{p_1+2}=0.$$
Since the zeros of the function $\sigma\mapsto \frac{a_0}{p_0+2} \sigma^{p_0+2}+\frac{a_1}{p_1+2}\sigma^{p_1+2}$ are isolated, $u(r)$ must converge to one of these zeros as $r$ goes to infinity. Combining with the fact that $u\in L^2$, we obtain
$$\lim_{r\to\infty} u(r)=0.$$
By uniqueness and since $F(r)=0$, we see that if $u(r)=0$ for some $r$, then $u$ is identically $0$. This implies that $u(r)$ has a constant sign. Changing $u$ into $-u$ if necessary, we can assume that $u(r)>0$ for large $r$. Using the equation satisfies by $u$ and the fact that $u$ goes to $0$ we obtain that $u''(r)>0$ for large $r$. Thus $u'(r)<0$ for large $r$. Using the equation $F(r)=0$, we deduce that for large $r$,
$$ \frac{1}{\sqrt{2}}u'(r)=-\sqrt{\frac{a_0}{p_0+2} |u(r)|^{p_0+2}+\frac{a_1}{p_1+2}|u(r)|^{p_1+2}}.$$
Solving this equation, we obtain that $u(r)r^{\frac{2}{p_1}}$ has a nonzero limit as $r$ goes to infinity. Since $p_1>4$, this contradicts $u\in L^2(\R)$, concluding the proof in the case $d=1$.

We next assume $d=2$ and let $M(r)=\sup_{s\geq r}|u(s)|$. Then by \eqref{eq:ell}, for large $r$, we have
$$|\partial_r(r\partial_r u)|\lesssim M(r)^{p_1-1}|u|^2r.$$
Integrating between $r=R$ and $\infty$, using that $M(r)$ is decreasing, that $\int_0^{\infty} |u|^2rdr<\infty$, and \eqref{goes_to_0} we obtain
$$ |\partial_ru(R)|\lesssim \frac{1}{R}M(R)^{p_1-1}.$$
Since $M(R)\lesssim R^{-1/2}$, we deduce, letting $\alpha=p_1-2>0$, that for $r$ large, 
$$|\partial_ru(r)|\lesssim M(r) \frac{1}{r^{1+\alpha/2}},$$
and thus integrating again between $r=R$ and $\infty$,
$$ |u(r)|\lesssim M(r)\frac{1}{r^{\alpha/2}}.$$
Since $r\mapsto M(r)\frac{1}{r^{\alpha/2}}$ is decreasing, we obtain
$$ M(r)\lesssim M(r)\frac{1}{r^{\alpha/2}}$$
for large $r$. Thus $M(r)=0$ for large $r$. This implies $u(r)=0$ for large $r$ and thus, by standard ODE theory, that $u$ is identically $0$.

We next assume $d\in \{3,4\}$. 
We first show that for all $a>0$, there exists a constant $C>0$ such that 
\begin{equation}
\label{decay}
\exists C>0,\quad r>C\Longrightarrow  |u(r)|\leq Cr^{-a}.
\end{equation} 
Indeed, since $|r^{d-1}\partial_ru|$ and $u(r)$ go to $0$ as $r\to\infty$, \eqref{decay} holds for $a=d-2$. Next, assuming that \eqref{decay} holds for some $a\geq d-2$, one has, integrating \eqref{eq:ell} between $r$ and $\infty$, $|r^{d-1}\partial_ru|\lesssim r^{d-(p_1+1)a}$, and thus $|u(r)|\lesssim r^{2-(p_1+1)a}$. This shows that \eqref{decay} holds with $a$ replaced by $a'=(p_1+1)a-2\geq a+ap_1-2$. Since $ap_1-2\geq p_1-2>0$, we deduce that \eqref{decay} holds for all $a>0$. Integrating twice the equation \eqref{eq:ell} as above we obtain that $M(r)\lesssim \frac{1}{r}M(r)$ for large $r$, which shows again that $u$ is identically $0$.
\end{proof}

\bibliographystyle{abbrv}
\bibliography{paper}
  
\end{document}